\documentclass[a4paper,12pt]{article}

\usepackage{authblk}
\usepackage{parskip}
\usepackage{mathtools}
\usepackage{amssymb}
\usepackage{amsmath}
\usepackage{latexsym}
\usepackage{mathrsfs}  
\usepackage{bbm}       
\usepackage{amsthm}    
\usepackage{relsize}   
\usepackage{thmtools}  
\usepackage{mathrsfs}  
\usepackage{lipsum}    
\usepackage{hyperref}
\usepackage{orcidlink}

\numberwithin{equation}{section}

\usepackage{enumitem}  
\setlist{nosep}  
\setitemize[0]{leftmargin=*}
\setenumerate[0]{leftmargin=*}
\setenumerate[1]{label={(\arabic*)}} 

\newcommand{\N}{\mathbb{N}}     
\newcommand{\R}{\mathbb{R}}     
\newcommand{\Prob}{\mathbb{P}}  
\newcommand{\Exp}{\mathbb{E}}   
\newcommand{\inner}[2]{\left\langle #1 \, , \, #2 \right\rangle} 
\newcommand{\norm}[1]{\left|\left|#1\right|\right|}              
\newcommand{\triplet}[3]{\left( #1, #2, #3 \right) }             
\newcommand{\ProbSpace}{\triplet{\Omega}{\mathscr{F}}{\Prob}}    
\newcommand{\abs}[1]{\left| #1 \right|}                          

\newcommand{\defeq}{\mathrel{\mathop:}=}                         
\newcommand\restr[2]{{
  \left.\kern-\nulldelimiterspace 
  #1 
  \vphantom{\big|} 
  \right|_{#2} 
  }}
  
\makeatletter
\newsavebox{\@brx}
\newcommand{\llangle}[1][]{\savebox{\@brx}{\(\m@th{#1\langle}\)}%
  \mathopen{\copy\@brx\kern-0.5\wd\@brx\usebox{\@brx}}}
\newcommand{\rrangle}[1][]{\savebox{\@brx}{\(\m@th{#1\rangle}\)}%
  \mathclose{\copy\@brx\kern-0.5\wd\@brx\usebox{\@brx}}}
\makeatother

\theoremstyle{plain} 
\newtheorem{theorem}{Theorem}[section]    
\newtheorem{proposition}[theorem]{Proposition} 
\newtheorem{corollary}[theorem]{Corollary}
\newtheorem{lemma}[theorem]{Lemma}

\theoremstyle{definition} 
\newtheorem{definition}[theorem]{Definition}
\newtheorem{example}[theorem]{Example}
\newtheorem{remark}[theorem]{Remark}

 \title{\Large Convergence Uniform on Compacts in Probability with Applications to Stochastic Analysis in Duals of Nuclear Spaces}
\author{C. A. Fonseca-Mora \orcidlink{0000-0002-9280-8212}}
\affil{Centro de Investigaci\'{o}n en Matem\'{a}tica Pura y Aplicada, \\ Escuela de Matem\'{a}tica, Universidad de Costa Rica. \\
\noindent E-mail:  christianandres.fonseca@ucr.ac.cr }
\date{}

\begin{document}

 \maketitle

\abstract{Let $\Phi'$ denote the strong dual of a nuclear space $\Phi$. In this paper we introduce sufficient conditions for the convergence uniform on compacts in probability for a sequence of $\Phi'$-valued processes with continuous or c\`{a}dl\`{a}g paths. 
We illustrate the usefulness of our results by considering two applications to stochastic analysis. First, we introduce a topology on the space of $\Phi'$-valued semimartingales which are good integrators and show that this topology is complete and that the stochastic integral mapping is continuous on the integrators. Second, we introduce sufficient conditions for the convergence uniform on compacts in probability of the solutions to a sequence of  linear stochastic evolution equations driven by semimartingale noise. }

\smallskip

\emph{2020 Mathematics Subject Classification:} 60B11, 60G48, 60H05, 60H15.

\emph{Key words and phrases:} convergence uniform on compacts in probability, semimartingales, stochastic integral, dual of a nuclear space.

\section{Introduction}


Let $\Phi$ be a nuclear space with strong dual $\Phi'$ and canonical pairing $\inner{f}{\phi}$ for $f \in \Phi'$ and $\phi \in \Phi$.  A $\Phi'$-valued stochastic process $(X_{t}:t \geq 0)$ is a semimartingale if  for each $\phi \in \Phi$ the real-valued process $\inner{X}{\phi}=(\inner{X_{t}}{\phi}: t \geq 0)$ is a semimartingale. The study of properties of $\Phi'$-valued semimartingales and stochastic calculus with respect to them has been explored in many works, see e.g.  \cite{BojdeckiGorostiza:1991, DawsonGorostiza:1990, FonsecaMora:Semimartingales, PerezAbreu:1988, PerezAbreuTudor:1992, PerezAbreuRochaArteagaTudor:2005, Ustunel:1982, Ustunel:1982-1, Ustunel:1986}.

In \cite{FonsecaMora:StochInteg} the author introduced a theory of real-valued stochastic integration with respect to semimartingales in duals of nuclear spaces. 
Indeed, if $X$ is a $\Phi'$-valued semimartingale and $H$ is a $\Phi$-valued process which is predictable and bounded, it is shown in \cite{FonsecaMora:StochInteg} that there corresponds a real-valued semimartingale $\int \, H \, dX$ called the \emph{stochastic integral}, and the \emph{stochastic integral mapping} $H \mapsto \int \, H \, dX$ is linear.
Whenever it happens that the mapping $H \mapsto \int \, H \, dX$ is continuous from the space of $\Phi$-valued bounded predictable processes into the space $S^{0}$ of real-valued semimartingales (with Emery's topology), we say that $X$ is a \emph{good integrator}. 

It was shown in Section 5 in \cite{FonsecaMora:StochInteg} that if $X$ is a good integrator, the stochastic integral possesses some important properties like the existence of a Riemann representation formula, a stochastic integration by parts formula, and a stochastic Fubini theorem.

Besides, in \cite{FonsecaMora:StochInteg} some other deeper properties of the stochastic integral were not studied; as are for example the continuity of the stochastic integral mapping on the integrators, the existence of a quadratic variation, the establishment of an It\^{o} formula, or the existence of a  vector-valued extension for the theory of stochastic integration. In order to prove these properties one requires to have a topology for the space of good integrators, as well to have a  better understanding of the convergence uniform  in probability on compact intervals of time (abreviated as UCP convergence). 

Motivated by the above, in this paper we introduce sufficient conditions for the UCP convergence of a sequence of $\Phi'$-valued processes. To illustrate the usefulness of our results on UCP convergence, in this article we apply them to the construction of a topology on the space of good integrators which makes the stochastic integral mapping continuous in the integrators, and to the UCP convergence of solutions to a sequence of stochastic evolution equations with semimartingale noise. These results will be applied in subsequent publications to prove further properties of the stochastic integral with respect to $\Phi'$-valued semimartingales, as for example are the properties described in the above paragraph. 

In the next few paragraphs we describe our results. In Section \ref{sectionPrelim} we list some important notions concerning nuclear spaces and their duals, cylindrical and stochastic processes, and the topology of uniform convergence on the space of continuous mappings. In Section \ref{sectUCP} our main result is Theorem \ref{theoUCPConver} where we provide sufficient conditions for the UCP convergence of stochastic processes in $\Phi'$ with continuous paths. Roughly speaking, if $(X^{n}:n \in \N)$ is a sequence of $\Phi'$-valued processes, under standard assumptions for $X^{n}$, we show that if $\inner{X^{n}}{\phi}$ converges in UCP to $\inner{X}{\phi}$ for every $\phi \in \Phi$, then $X^{n}$ converges to $X$ in UCP as $\Phi'$-valued processes. 
To the extent of our knowledge this result has not been considered in the literature, not even for particular classes of nuclear spaces. A specialized version of Theorem \ref{theoUCPConver} for the case when $\Phi$ is ultrobornological is given in Section \ref{subSecUltrobornoUCP}, and applications of Theorem \ref{theoUCPConver} to UCP convergence of $\Phi'$-valued  semimartingales are given in Section \ref{subSectSemimar}. 


The main goal of Section \ref{subSectGoodInteg}  is the introduction of a topology of the space of good integrators and to show some of its properties. We start in Section \ref{subSecRealStochInteg} with a review of the rudiments of the construction and properties of the stochastic integral introduced in  \cite{FonsecaMora:StochInteg}. Next, in Section \ref{subSectGoodIntegFrechet} under the assumption that $\Phi$ is a Fr\'{e}chet nuclear space we introduce a topology for the space of good integrators and show that it is linear, metrizable and complete  (see Theorem \ref{theoCompleteSpaceIntegrators}). As a consequence we show that the bilinear form defined by the stochastic integral is continuous (Corollary \ref{coroContiStochIntegralMapping}). We further show that the collection of those good integrators in $\Phi'$ that are continuous local martingales is a closed subspace (Proposition \ref{propContLocaMartiClosedSubspace}).   Finally, in Section \ref{subSectGoodIntegInducLimitFrechet} we extend our definition of the topology on the space of good integrators to the case where  $\Phi$ is a strict inductive limit of a sequence of Fr\'{e}chet nuclear spaces. The resulting topology is also linear, metrizable and complete, and the bilinear form defined by the stochastic integral is continuous (Corollary \ref{coroContiBilinearMappingLFNuclearSpace}).
We are not aware of any work that considers the introduction of a topology on the space of semimartingales taking values in duals of nuclear spaces. 

In Section \ref{sectUCPConvSPDEs} we apply the tools developed in Section \ref{sectUCP} to provide sufficient conditions for the UCP convergence of the solutions to the sequence of linear stochastic evolution equations 
$$ dY^{n}=(A^{n})'Y^{n}_{t} dt + dZ^{n}_{t}, $$
where for each $n \in \N$, $A^{n}$ is the generator to a $C_{0}$-semigroup of continuous linear operators on $\Phi$, and $Z^{n}$ is a $\Phi'$-valued semimartingale. Existence and uniqueness of solutions to such stochastic evolution equations were studied by the author in \cite{FonsecaMora:StochInteg}.

\section{Preliminaries} \label{sectionPrelim}

\subsection{Nuclear spaces and their strong duals}

Let $\Phi$ be a locally convex space (we will only consider vector spaces over $\R$). $\Phi$ is called \emph{bornological} (respectively,  \emph{ultrabornological}) if it is the inductive limit of a family of normed (respectively, Banach) spaces. A \emph{barreled space} is a locally convex space such that every convex, balanced, absorbing and closed subset is a neighborhood of zero. For further details see \cite{Jarchow, NariciBeckenstein}.   

If $p$ is a continuous semi-norm on $\Phi$ and $r>0$, the closed ball of radius $r$ of $p$, given by $B_{p}(r) = \left\{ \phi \in \Phi: p(\phi) \leq r \right\}$, is a closed, convex, balanced neighborhood of zero in $\Phi$. 
A continuous seminorm (respectively norm) $p$ on $\Phi$ is called \emph{Hilbertian} if $p(\phi)^{2}=Q(\phi,\phi)$, for all $\phi \in \Phi$, where $Q$ is a symmetric non-negative bilinear form (respectively inner product) on $\Phi \times \Phi$. For any given continuous seminorm $p$ on $\Phi$ let $\Phi_{p}$ be the Banach space that corresponds to the completion of the normed space $(\Phi / \mbox{ker}(p), \tilde{p})$, where $\tilde{p}(\phi+\mbox{ker}(p))=p(\phi)$ for each $\phi \in \Phi$. We denote by  $\Phi'_{p}$ the dual to the Banach space $\Phi_{p}$ and by $p'$ the corresponding dual norm. Observe that if $p$ is Hilbertian then $\Phi_{p}$ and $\Phi'_{p}$ are Hilbert spaces. If $q$ is another continuous seminorm on $\Phi$ for which $p \leq q$, we have that $\mbox{ker}(q) \subseteq \mbox{ker}(p)$ and the canonical inclusion map from $\Phi / \mbox{ker}(q)$ into $\Phi / \mbox{ker}(p)$ has a unique continuous and linear extension that we denote by $i_{p,q}:\Phi_{q} \rightarrow \Phi_{p}$. Furthermore, we have the following relation: $i_{p}=i_{p,q} \circ i_{q}$.

We denote by $\Phi'$ the topological dual of $\Phi$ and by $\inner{f}{\phi}$ the canonical pairing of elements $f \in \Phi'$, $\phi \in \Phi$. Unless otherwise specified, $\Phi'$ will always be consider equipped with its \emph{strong topology}, i.e. the topology on $\Phi'$ generated by the family of semi-norms $( \eta_{B} )$, where for each $B \subseteq \Phi$ bounded, $\eta_{B}(f)=\sup \{ \abs{\inner{f}{\phi}}: \phi \in B \}$ for all $f \in \Phi'$.  

Let $p$ and $q$ be continuous Hilbertian semi-norms on $\Phi$ such that $p \leq q$. The space of continuous linear operators (respectively Hilbert-Schmidt operators) from $\Phi_{q}$ into $\Phi_{p}$ is denoted by $\mathcal{L}(\Phi_{q},\Phi_{p})$ (respectively $\mathcal{L}_{2}(\Phi_{q},\Phi_{p})$). We employ an analogous notation for operators between the dual spaces $\Phi'_{p}$ and $\Phi'_{q}$. 

Let us recall that a (Hausdorff) locally convex space $(\Phi,\mathcal{T})$ is called \emph{nuclear} if its topology $\mathcal{T}$ is generated by a family $\Pi$ of Hilbertian semi-norms such that for each $p \in \Pi$ there exists $q \in \Pi$, satisfying $p \leq q$ and the canonical inclusion $i_{p,q}: \Phi_{q} \rightarrow \Phi_{p}$ is Hilbert-Schmidt. Other equivalent definitions of nuclear spaces can be found in \cite{Pietsch, Treves}. 

Let $\Phi$ be a nuclear space. If $p$ is a continuous Hilbertian semi-norm  on $\Phi$, then the Hilbert space $\Phi_{p}$ is separable (see \cite{Pietsch}, Proposition 4.4.9 and Theorem 4.4.10, p.82). Now, let $( p_{n} : n \in \N)$ be an increasing sequence of continuous Hilbertian semi-norms on $(\Phi,\mathcal{T})$. We denote by $\theta$ the locally convex topology on $\Phi$ generated by the family $( p_{n} : n \in \N)$. The topology $\theta$ is weaker than $\mathcal{T}$. We  will call $\theta$ a (weaker) \emph{countably Hilbertian topology} on $\Phi$ and we denote by $\Phi_{\theta}$ the space $(\Phi,\theta)$ and by $\widehat{\Phi}_{\theta}$ its completion. The space $\widehat{\Phi}_{\theta}$ is a (not necessarily Hausdorff) separable, complete, pseudo-metrizable (hence Baire and ultrabornological; see Example 13.2.8(b) and Theorem 13.2.12 in \cite{NariciBeckenstein}) locally convex space and its dual space satisfies $(\widehat{\Phi}_{\theta})'=(\Phi_{\theta})'=\bigcup_{n \in \N} \Phi'_{p_{n}}$ (see \cite{FonsecaMora:Existence}, Proposition 2.4).

\subsection{Cylindrical and stochastic processes}
Let $E$ be a topological space and denote by $\mathcal{B}(E)$ its Borel $\sigma$-algebra. Recall that a Borel measure $\mu$ on $E$ is called a \emph{Radon measure} if for every $\Gamma \in \mathcal{B}(E)$ and $\epsilon >0$, there exist a compact set $K \subseteq \Gamma$ such that $\mu(\Gamma \backslash K) < \epsilon$. 

Throughout this work we assume that $\ProbSpace$ is a complete probability space and consider a filtration $( \mathcal{F}_{t} : t \geq 0)$ on $\ProbSpace$ that satisfies the \emph{usual conditions}, i.e. it is right continuous and $\mathcal{F}_{0}$ contains all subsets of sets of $\mathcal{F}$ of $\Prob$-measure zero. The space $L^{0} \ProbSpace$ of equivalence classes of real-valued random variables defined on $\ProbSpace$ will always be considered equipped with the topology of convergence in probability and in this case it is a complete, metrizable, topological vector space.

Let $\Phi$ be a locally convex space. A \emph{cylindrical random variable}\index{cylindrical random variable} in $\Phi'$ is a linear map $X: \Phi \rightarrow L^{0} \ProbSpace$ (see \cite{FonsecaMora:Existence}). Let $X$ be a $\Phi'$-valued random variable, i.e. $X:\Omega \rightarrow \Phi'$ is a $\mathscr{F}/\mathcal{B}(\Phi')$-measurable map. For each $\phi \in \Phi$ we denote by $\inner{X}{\phi}$ the real-valued random variable defined by $\inner{X}{\phi}(\omega) \defeq \inner{X(\omega)}{\phi}$, for all $\omega \in \Omega$. The linear mapping $\phi \mapsto \inner{X}{\phi}$ is called the \emph{cylindrical random variable induced/defined by} $X$.


Let $J=\R_{+} \defeq [0,\infty)$ or $J=[0,T]$ for  $T>0$. We say that $X=( X_{t}: t \in J)$ is a \emph{cylindrical process} in $\Phi'$ if $X_{t}$ is a cylindrical random variable for each $t \in J$. Clearly, any $\Phi'$-valued stochastic processes $X=( X_{t}: t \in J)$ induces/defines a cylindrical process under the prescription: $\inner{X}{\phi}=( \inner{X_{t}}{\phi}: t \in J)$, for each $\phi \in \Phi$. 

If $X$ is a cylindrical random variable in $\Phi'$, a $\Phi'$-valued random variable $Y$ is called a \emph{version} of $X$ if for every $\phi \in \Phi$, $X(\phi)=\inner{Y}{\phi}$ $\Prob$-a.e. A $\Phi'$-valued process $Y=(Y_{t}:t \in J)$ is said to be a $\Phi'$-valued \emph{version} of the cylindrical process $X=(X_{t}: t \in J)$ on $\Phi'$ if for each $t \in J$, $Y_{t}$ is a $\Phi'$-valued version of $X_{t}$.

A $\Phi'$-valued process $X=( X_{t}: t \in J)$ is  \emph{continuous} (respectively \emph{c\`{a}dl\`{a}g}) if for $\Prob$-a.e. $\omega \in \Omega$, the \emph{sample paths} $t \mapsto X_{t}(\omega) \in \Phi'$ of $X$ are continuous (respectively c\`{a}dl\`{a}g). 


A $\Phi'$-valued random variable $X$ is called \emph{regular} if there exists a weaker countably Hilbertian topology $\theta$ on $\Phi$ such that $\Prob( \omega: X(\omega) \in (\widehat{\Phi}_{\theta})')=1$. If $\Phi$ is a barrelled (e.g. ultrabornological) nuclear space, the property of being regular is  equivalent to the property that the law of $X$ be a Radon measure on $\Phi'$ (see Theorem 2.10 in \cite{FonsecaMora:Existence}).  A $\Phi'$-valued process $X=(X_{t}:t \geq 0)$ is said to be \emph{regular} if for each $t \geq 0$, $X_{t}$ is a regular random variable. 

\subsection{Topology of UCP convergence}\label{subSectDefiUCP}

Let $\Phi$ be a Hausdorff locally convex space and let $\Pi$ denotes a system of seminorms generating the  topology on $\Phi'$.
Denote by $\mathbbm{C}(\Phi')$ the linear space of all the $\Phi'$-valued continuous processes $X=(X_{t}: t \geq 0)$.

The \emph{topology of convergence uniformly on compacts in probability} (abbreviated as \emph{UCP}) on $\mathbbm{C}(\Phi')$  is the linear topology whose fundamental system of neighborhoods of zero for this topology is the family of all the sets of the form 
$$\left\{ X \in \mathbbm{C}(\Phi') : \Prob \left( \sup_{0 \leq t \leq T} p(X_{t}) \geq \epsilon \right) < \eta \right\}, \quad \, \forall \, \epsilon, \eta, T>0, \, p \in \Pi. $$

 Let $X$ and $(X^{n}: n \in \N)$, with $X^{n} =(X^{n}_{t}: t \geq 0)$, be $\Phi'$-valued continuous processes. 
 We say that $X^{n}$ converges to $X$ \emph{uniformly on compacts in probability}, abbreviated as $X^{n} \overset{ucp}{\rightarrow} X$, if $X^{n}$ converges to $X$ in the UCP topology, i.e.  if for every choice of $T>0$, $\epsilon >0$, and every continuous seminorm $p$ on $\Phi'$ we have 
$$ \lim_{n \rightarrow \infty} \Prob \left( \sup_{0 \leq t \leq T} p(X^{n}_{t}-X_{t}) \geq \epsilon \right)=0.  $$



As the next result shows UCP convergence for $\Phi'$-valued continuous processes imply that the finite dimensional projections of the sequence converges in UCP.

\begin{proposition}\label{propUCPFiniteDimensionProjections}
Let $X=(X_{t}: t \geq 0)$ and  $X^{n} =(X^{n}_{t}: t \geq 0)$, $n \in \N$, be $\Phi'$-valued processes with continuous  paths. If $X^{n} \overset{ucp}{\rightarrow} X$, then for every $m \in \N$, $\phi_{1}, \cdots, \phi_{m} \in \Phi$, we have 
$$\left(\inner{X^{n}}{\phi_{1}}, \cdots, \inner{X^{n}}{\phi_{m}} \right) \overset{ucp}{\rightarrow} \left( \inner{X}{\phi_{1}}, \cdots, \inner{X}{\phi_{m}} \right).$$
\end{proposition}
\begin{proof} It suffices to check that for each $\phi \in \Phi$ we have $\inner{X^{n}}{\phi} \overset{ucp}{\rightarrow} \inner{X}{\phi}$. Let $\phi \in \Phi$. The set $\{\phi\}$ is bounded in $\Phi$, then $p_{\{\phi\}}(f)=\abs{\inner{f}{\phi}}$ $\forall f \in \Phi'$ is a continuous seminorm on $\Phi'$. Since $X^{n} \overset{ucp}{\rightarrow} X$, we have for any $T>0$, $\epsilon >0$, 
$$ \lim_{n \rightarrow \infty} \Prob \left( \sup_{0 \leq t \leq T} \abs{\inner{X^{n}_{t}}{\phi} - \inner{X_{t}}{\phi}} \geq \epsilon \right)=\lim_{n \rightarrow \infty} \Prob \left( \sup_{0 \leq t \leq T} p_{\phi}(X^{n}_{t}-X_{t}) \geq \epsilon \right) = 0. $$
\end{proof}



We denote by $C_{\infty}(\Phi')$ the collection of all continuous mappings $x:[0,\infty) \rightarrow \Phi'$. We equip $C_{\infty}(\Phi')$ with the topology of \emph{uniform convergence on compact intervals of time}, i.e. if $\Pi$ is any generating family of seminorms for the topology on $\Phi'$, the topology on $C_{\infty}(\Phi')$ is the locally convex topology generated by the family of seminorms
$$ x \mapsto \sup_{t \in [0,T]} p(x(t)), \quad \forall \, p \in \Pi, \, T >0. $$
If $X$ is a $C_{\infty}(\Phi')$-valued random variable, then $X$  determines a $\Phi'$-valued continuous process.

Now we explore the relation between UCP convergence and weak convergence in $C_{\infty}(\Phi')$.
Assume  $X$ and each $X^{n}$ is a random variable in $C_{\infty}(\Phi')$. Then it is clear that $X^{n} \overset{ucp}{\rightarrow} X$ if and only if $X^{n} \rightarrow X$ in probability as elements in $C_{\infty}(\Phi')$. In any of these cases we have  $X^{n} \Rightarrow X$ in $C_{\infty}(\Phi')$. As expected the converse holds if the limit is deterministic, since we will need this result latter we include a full proof for this fact.  

\begin{lemma}\label{lemmaWeakImpliesUCP}
If $Y^{n}$ is a sequence of $C_{\infty}(\Phi')$-valued random variables satisfying $Y^{n} \Rightarrow 0$ in $C_{\infty}(\Phi')$, then $Y^{n} \overset{ucp}{\rightarrow} 0$. 
\end{lemma}
\begin{proof}
For each $n \in \N$, let $\mu_{n}$ be the distribution of $Y^{n}$ on $C_{\infty}(\Phi')$. Choose any $T>0$, $\epsilon >0$, and a continuous seminorm $p$ on $\Phi'$.  Let $q_{p,T}(x)=\sup_{t \in [0,T]} p(x(t))$ $\forall x \in C_{\infty}(\Phi')$, which is a continuous seminorm on $C_{\infty}(\Phi')$.

By Alexandrov's Theorem (see Theorem I.3.5 in \cite{VakhaniaTarieladzeChobanyan}, p.42), since $Y^{n} \Rightarrow 0$ in $C_{\infty}(\Phi')$ we have
$$
\lim_{n \rightarrow \infty} \Prob \left( \sup_{t \in [0,T]} p(Y^{n}_{t}) \geq \epsilon \right) 
= \lim_{n \rightarrow \infty} \mu_{n} \left( B_{q_{p,T}}(\epsilon)^{c} \right) \leq  \delta_{0} \left( B_{q_{p,T}}(\epsilon)^{c} \right) =0,
$$
where $\delta_{0}$ is the Dirac measure on $C_{\infty}(\Phi')$. Then $Y^{n} \overset{ucp}{\rightarrow} 0$. 
\end{proof}

The following result,  proved in \cite{FonsecaMora:WeakConverg}, introduces sufficient conditions for the weak convergence on $C_{\infty}(\Phi')$ for a sequence of (cylindrical) stochastic processes in $\Phi'$ for a nuclear space $\Phi$. It will play a major role in our proof  of Theorem \ref{theoUCPConver} in Section \ref{subSectSuffiCondiUCPConvergence}.

\begin{theorem}[\cite{FonsecaMora:WeakConverg}, Theorem 5.1] \label{theoWeakConvergCylinProcesInCTInfy}
Let $\Phi$ be a nuclear space and let $(X^{n}: n \in \N)$, with $X^{n} =(X^{n}_{t}: t \geq 0)$, be a sequence of cylindrical process in $\Phi'$ satisfying:
\begin{enumerate}
\item For every $n \in \N$ and $\phi \in \Phi$ the real-valued process $X^{n}(\phi)=( X^{n}_{t}(\phi): t \geq 0)$ is continuous.
\item For every $T > 0$, the family $( X^{n}_{t}: t \in [0,T], n \in \N )$ of linear maps from $\Phi$ into $L^{0} \ProbSpace$ is equicontinuous at zero.   
\item For each $\phi \in \Phi$, the sequence of distributions of $X^{n}(\phi)$ is uniformly tight on $C_{\infty}(\R)$.
\item $\forall$ $m \in \N$, $\phi_{1}, \dots, \phi_{m} \in \Phi$, $t_{1}, \dots, t_{m} \geq 0$, the distribution of the $m$-dimensional random vector $(X_{t_{1}}^{n}(\phi_{1}), \cdots, X_{t_{m}}^{n}(\phi_{m}))$ converges in distribution to some probability measure on $\R^{m}$. 
\end{enumerate}
Then there exist a weaker countably Hilbertian topology $\theta$ on $\Phi$ and some  $C_{\infty}((\widehat{\Phi}_{\theta})')$-valued random variables  $Y$ and  $Y^{n}$, $n \in \N$,  such that 
\begin{enumerate}[label=(\roman*)]
\item For every $\phi \in \Phi$ and $n \in \N$, the real-valued processes $\inner{Y^{n}}{\phi}$ and $X^{n}(\phi)$ are indistinguishable.
\item The sequence $(Y_{n}:n \in \N)$ is tight on $C_{\infty}((\widehat{\Phi}_{\theta})')$.
\item $Y^{n} \Rightarrow Y$ in $C_{\infty}((\widehat{\Phi}_{\theta})')$
\end{enumerate}
Moreover, (ii) and (iii) are also satisfied for $Y$ and $(Y^{n}: n \in \N)$ as $C_{\infty}(\Phi')$-valued random variables. 
\end{theorem}

\section{Sufficient conditions for UCP convergence}\label{sectUCP}

\subsection{Main Result}\label{subSectSuffiCondiUCPConvergence}

Under the assumption that $\Phi$ is a nuclear space, we provide in this section sufficient conditions for UCP convergence of a sequence of  $\Phi'$-valued processes with continuous paths. Our main result is formulated in the more general  context of cylindrical processes.

\begin{theorem}\label{theoUCPConver}
Let $\Phi$ be a nuclear space and let $(X^{n}: n \in \N)$, with $X^{n} =(X^{n}_{t}: t \geq 0)$, be a sequence of cylindrical process in $\Phi'$ satisfying:
\begin{enumerate}
\item \label{hypothesisContiTestFunct} For every $n \in \N$ and $\phi \in \Phi$, the real-valued process $X^{n}(\phi)=( X^{n}_{t}(\phi): t \geq 0)$ is continuous.
\item For every $n \in \N$ and $T > 0$, the family $( X^{n}_{t}: t \in [0,T] )$ of linear maps from $\Phi$ into $L^{0} \ProbSpace$ is equicontinuous.  
\item \label{hypothesisWeakUCP} For every $\phi \in \Phi$, the sequence $X^{n}(\phi)$ converges uniformly on compacts in probability.
\end{enumerate}
Then, there exist a weaker countably Hilbertian topology $\vartheta$ on $\Phi$ and some  $(\widehat{\Phi}_{\vartheta})'$-valued continuous processes  $Y= (Y_{t}: t \geq 0)$ and  $Y^{n}= (Y^{n}_{t}: t \geq 0)$, $n \in \N$,  such that 
\begin{enumerate}[label=(\roman*)]
\item \label{mainTheoConclu1} For every $\phi \in \Phi$ and $n \in \N$, the real-valued processes $\inner{Y^{n}}{\phi}$ and $X^{n}(\phi)$ are indistinguishable.
\item \label{mainTheoConclu2} $Y^{n} \overset{ucp}{\rightarrow} Y$  as  $(\widehat{\Phi}_{\vartheta})'$-valued processes. 
\end{enumerate}
Moreover, as  $\Phi'$-valued processes $Y$  and $Y^{n}$, for $n \in \N$, are continuous processes and $Y^{n} \overset{ucp}{\rightarrow} Y$.  
\end{theorem}

\begin{remark}
Assume that in Theorem \ref{theoUCPConver} we have $(X^{n}: n \in \N)$ is a sequence of $\Phi'$-valued regular processes. Then, the conclusion in $\ref{mainTheoConclu1}$ implies that for each $n \in \N$, $Y^{n}$ is a continuous version of $X^{n}$. If instead of \ref{hypothesisContiTestFunct} we assume that each $X^{n}$ has continuous paths, then $X^{n}$ and $Y^{n}$ are indistinguishable processes (see Proposition 2.12 in \cite{FonsecaMora:Existence}). The conclusion in $\ref{mainTheoConclu2}$ therefore shows that $X^{n} \overset{ucp}{\rightarrow} Y$  as  $(\widehat{\Phi}_{\vartheta})'$-valued processes, thus the convergence occurs in a topology finer than that of $\Phi'$. 
\end{remark}

\begin{remark}
An analogue of Theorem \ref{theoUCPConver} for almost sure uniform convergence is proved in \cite{FonsecaMora:AlmostSure} (Theorem 3.1). Neither the result in \cite{FonsecaMora:AlmostSure} nor the arguments used in its proof can be used to show Theorem \ref{theoUCPConver}. This because the argument used in \cite{FonsecaMora:AlmostSure} relies heavily on the fact that for almost every $\omega \in \Omega$ one can show that the convergence uniform on a compact interval of time $[0,T]$ occurs in a Hilbert space $\Phi'_{q}$, where the continuous Hilbertian seminorm $q$ on $\Phi$ depends on $\omega$ and $T>0$.  
\end{remark}

Our proof of Theorem \ref{theoUCPConver} will follow an argument based on establishing weak convergence on $C_{\infty}(\Phi')$ for $Y^{n}-Y$ to $0$ and then applying Lemma \ref{lemmaWeakImpliesUCP}. 
For this argument to work we will need to show that our hypothesis \ref{hypothesisWeakUCP} in Theorem \ref{theoUCPConver} on weak (in the duality sense) ucp convergence implies weak convergence on $C_{\infty}(\Phi')$. This task will be carried out with the help of Theorem \ref{theoWeakConvergCylinProcesInCTInfy}.



\begin{proof}[Proof of Theorem \ref{theoUCPConver}]
We divide the proof if three steps. 

\textbf{Step 1:} \emph{There exists a weaker countably Hilbertian topology $\theta$ and a sequence of $(\widehat{\Phi}_{\theta})'$-valued continuous processes  $(Y^{n}: n \in \N)$ satisfying (i) and such that for each $T>0$ the family $(Y^{n}_{t}: t \in [0,T], n \in \N)$ of linear maps from $\widehat{\Phi}_{\theta}$ into $L^{0}\ProbSpace$ is equicontinuous at zero.}

In effect, for every $n \in \N$ we have by assumptions (1) and (2) and the regularization theorem (Theorem 3.2 in \cite{FonsecaMora:Existence}) that there exist a weaker countably Hilbertian topology $\theta_{n}$ on $\Phi$ and a $(\widehat{\Phi}_{\theta_{n}})'$-valued continuous process $Y^{n}=(Y^{n}_{t}: t \geq 0 )$ such that $\inner{Y^{n}}{\phi}$ and $X^{n}(\phi)$ are indistinguishable processes for every $\phi \in \Phi$. 

Given $t \geq 0$ and $n \in \N$, we must show that the linear mapping $Y^{n}_{t}: \widehat{\Phi}_{\theta_{n}} \rightarrow L^{0}\ProbSpace$ is continuous. Since $\widehat{\Phi}_{\theta_{n}}$ is Baire and pseudo-metrizable, by the closed graph theorem (see Theorems 14.1.1 and 14.3.4 in \cite{NariciBeckenstein}) it is enough to show that $Y^{n}_{t}$ is sequentially closed. Let $(\phi_{k}: k \in \N)$ be a sequence converging to $\phi$ in  $\widehat{\Phi}_{\theta_{n}}$, and assume that there exists $Z_{t} \in L^{0}\ProbSpace$ such that $\abs{\inner{Y^{n}_{t}}{\phi_{k}}-Z_{t}} \rightarrow 0$ in probability as $k \rightarrow \infty$. Hence, there exists a subsequence $(\phi_{k(m)}: m \in \N)$ such that $\lim_{m \rightarrow \infty} \inner{Y^{n}_{t}}{\phi_{k(m)}}=Z_{t}$ $\Prob$-a.e. Since $Y^{n}_{t}$ is a $(\widehat{\Phi}_{\theta_{n}})'$-valued random variable, $\lim_{m \rightarrow \infty} \inner{Y^{n}_{t}}{\phi_{k(m)}}=\inner{Y^{n}_{t}}{\phi}$ $\omega$-wise. Then showing $ \inner{Y^{n}_{t}}{\phi}=Z_{t}$ $\Prob$-a.e. Thus $Y^{n}_{t}$ is  sequentially closed, hence continuous. 

Let $\theta$ denotes the countably Hilbertian topology on $\Phi$ generated by the families of seminorms generating the topologies $\theta_{n}$, $n \in \N$. By definition $\theta$ is weaker than the given topology on $\Phi$ and is finer than each $\theta_{n}$. Thus each $Y^{n}$ is a  
$(\widehat{\Phi}_{\theta})'$-valued continuous process and $Y^{n}_{t}$ is continuous as a linear operator from $\widehat{\Phi}_{\theta}$ into $L^{0}\ProbSpace$.
Moreover, by (3) we have for each $T>0$ and $\phi \in \Phi$ that  $\sup_{n} \sup_{0 \leq t \leq T} \abs{\inner{Y^{n}_{t}}{\phi}}= \sup_{n} \sup_{0 \leq t \leq T} \abs{X^{n}_{t}(\phi)}< \infty$ $\Prob$-a.e. 
Since  $\widehat{\Phi}_{\theta}$ is ultrabornological, by Proposition 3.3 in \cite{FonsecaMora:AlmostSure} for each $T>0$ and $\epsilon>0$ there exists a $\theta$-continuous Hilbertian seminorm $p$ on $\Phi$ such that 
\begin{equation}\label{eqSequenceYnEquicontFourTransforms}
\int_{\Omega} \, \sup_{n \in \N} \sup_{0 \leq t \leq T} \abs{1-e^{i\inner{Y^{n}_{t}}{\phi}}} \, d\Prob \leq \epsilon + 2 p(\phi), \quad \forall \, \phi \in \widehat{\Phi}_{\theta}. 
\end{equation}
By \eqref{eqSequenceYnEquicontFourTransforms} we conclude that for each $T>0$ the family $(Y^{n}_{t}: t \in [0,T], n \in \N)$ of linear maps from $\widehat{\Phi}_{\theta}$ (hence from $\Phi$) into $L^{0}\ProbSpace$ is equicontinuous at zero.

\textbf{Step 2:} \emph{There exists a weaker countably Hilbertian topology $\sigma$ on $\Phi$, finer than $\theta$ in Step 1, for which there exists a $(\widehat{\Phi}_{\sigma})'$-valued continuous process $Y$ such that $X^{n}(\phi) \overset{ucp}{\rightarrow} \inner{Y}{\phi}$  for each $\phi \in \Phi$, and for each $T>0$ the family $(Y_{t}: t \in [0,T])$ of linear maps from $\widehat{\Phi}_{\sigma}$  into $L^{0}\ProbSpace$ is equicontinuous at zero. }

In effect, for every $\phi \in \Phi$, by (1) and (3) there exists $X^{\phi} \in \mathbbm{C}(\R)$ such that $X^{n}(\phi) \overset{ucp}{\rightarrow} X^{\phi}$. Define the mapping 
$X: \Phi \rightarrow \mathbbm{C}(\R)$ by $X(\phi)=X^{\phi}$. By uniqueness of limits we can verify that $X$ is a linear mapping, hence defines a cylindrical process $X=(X_{t}: t \geq 0)$ in $\Phi$.
We are going to verify that for every $T>0$ the family of mappings $(X_{t}: t \in [0,T])$ are equicontinuous on $\widehat{\Phi}_{\theta}$ at the origin.

In effect, given $T>0$ and $\phi \in \Phi$, we have first by Fatou's lemma and the convergence in probability of $X^{n}(\phi)$ to $X(\phi)$ uniformly on $[0,T]$, then by the definition of $Y^{n}$ in Step 1  that  
\begin{eqnarray*}
\int_{\Omega} \,  \sup_{0 \leq t \leq T} \abs{1-e^{i X_{t}(\phi)}} d\Prob
& \leq & \liminf_{n \rightarrow \infty} \int_{\Omega} \,  \sup_{0 \leq t \leq T} \abs{1-e^{i X^{n}_{t}(\phi)}} d\Prob \\
& = & \liminf_{n \rightarrow \infty} \int_{\Omega} \,  \sup_{0 \leq t \leq T} \abs{1-e^{i\inner{Y^{n}_{t}}{\phi}}} \, d\Prob \\
& \leq & \int_{\Omega} \, \sup_{n \in \N} \sup_{0 \leq t \leq T} \abs{1-e^{i\inner{Y^{n}_{t}}{\phi}}} \, d\Prob. 
\end{eqnarray*} 
Thus from \eqref{eqSequenceYnEquicontFourTransforms} we conclude for each $T>0$ the equicontinuity on $\widehat{\Phi}_{\theta}$ (hence on $\Phi$) at the origin of the Fourier transforms of the family $(X_{t}: t \in [0,T])$. By \eqref{eqSequenceYnEquicontFourTransforms} the family  $(X_{t}: t \in [0,T])$ is equicontinuous on $\widehat{\Phi}_{\theta}$ at the origin. Therefore we can apply the regularization theorem (Theorem 3.2 in \cite{FonsecaMora:Existence})  to show the existence of a  weaker countably Hilbertian topology $\sigma$ on $\Phi$, finer than $\theta$ in Step 1, for which there exists a $(\widehat{\Phi}_{\sigma})'$-valued continuous process $Y=(Y_{t}: t \geq 0 )$ such that $\inner{Y}{\phi}$ and $X(\phi)$ are indistinguishable processes for every $\phi \in \Phi$ (thus $X^{n}(\phi) \overset{ucp}{\rightarrow} \inner{Y}{\phi}$). 


Finally, since $\sigma$ is finer than $\theta$, and for every $t \geq 0$ the Fourier transform of $X_{t}$ and $Y_{t}$ coincide, then 
for each $T>0$ the family $(Y_{t}: t \in [0,T], n \in \N)$ of linear maps from $\widehat{\Phi}_{\sigma}$ (hence from $\Phi$) into $L^{0}\ProbSpace$ is equicontinuous at zero. 

\textbf{Step 3:} \emph{There exists a weaker countably Hilbertian topology $\vartheta$, finer than $\sigma$ as given in Step 2,  such that $Y^{n} \overset{ucp}{\rightarrow} Y$  as  $(\widehat{\Phi}_{\vartheta})'$-valued processes.}

First, observe that since the topology $\sigma$ is finer than $\theta$, then the conclusions in Step 1 are valid replacing $\widehat{\Phi}_{\theta}$ with $\widehat{\Phi}_{\sigma}$. This way we can consider each $Y^{n}$ a $(\widehat{\Phi}_{\sigma})'$-valued continuous process and the equicontinuity of the family $(Y^{n}_{t}: t \in [0,T], n \in \N)$ is on $\widehat{\Phi}_{\sigma}$. 

For each $n \in \N$ define $Z^{n} = Y^{n}-Y$. We will check $(Z^{n}:n \in \N)$ satisfies the assumptions in Theorem \ref{theoWeakConvergCylinProcesInCTInfy}. From the discussion in the previous paragraph and Step 2 the assumptions (1) and (2) in Theorem \ref{theoWeakConvergCylinProcesInCTInfy} are satisfied, so we must check assumptions (3) and (4) in Theorem \ref{theoWeakConvergCylinProcesInCTInfy} holds true as well.

In effect, for any given $\phi \in \Phi$, since $X^{n}(\phi) \overset{ucp}{\rightarrow} Y(\phi)$ and the processes $X^{n}(\phi)$ and $\inner{Y^{n}}{\phi}$ are indistinguishable, then $\inner{Z^{n}}{\phi}=\inner{Y^{n}-Y}{\phi} \rightarrow 0$ in probability in $C_{\infty}(\R)$ and hence $\inner{Z^{n}}{\phi} \Rightarrow 0$ in $C_{\infty}(\R)$. Then by Prokhorov theorem the sequence of distributions of $\inner{Z^{n}}{\phi}$ is uniformly tight on $C_{\infty}(\R)$. 

Furthermore, since $\inner{Z^{n}}{\phi} \overset{ucp}{\rightarrow} 0$ for every $\phi \in \Phi$, we have that $\forall$ $m \in \N$, $\phi_{1}, \dots, \phi_{m} \in \Phi$, $t_{1}, \dots, t_{m} \geq 0$, the distribution of the $m$-dimensional random vector $(\inner{Z_{t_{1}}^{n}}{\phi_{1}}, \cdots, \inner{Z_{t_{m}}^{n}}{\phi_{m}})$ converges in distribution to the dirac measure $\delta_{0}$ at $0$ on $\R^{m}$. 

Then by Theorem \ref{theoWeakConvergCylinProcesInCTInfy} there exists a weaker countably Hilbertian topology $\vartheta$ on $\Phi$, which we can choose finer than $\sigma$, and some  $C_{\infty}((\widehat{\Phi}_{\vartheta})')$-valued random variables  $\tilde{Z}$ and  $\tilde{Z}^{n}$, $n \in \N$,  such that 
\begin{enumerate}[label=(\Roman*)]
\item For every $\phi \in \Phi$ and $n \in \N$, the real-valued processes $\inner{\tilde{Z}^{n}}{\phi}$ and $\inner{Z^{n}}{\phi}$ are indistinguishable.
\item $\tilde{Z}^{n} \Rightarrow \tilde{Z}$ in $C_{\infty}((\widehat{\Phi}_{\vartheta})')$.
\end{enumerate} 
By (I) and Proposition 2.12 in \cite{FonsecaMora:Existence} we have $\tilde{Z}^{n}$ and $Z^{n}$ are indistinguishable as $(\widehat{\Phi}_{\vartheta})'$-valued processes, thus they can be identified as random variables in $C_{\infty}((\widehat{\Phi}_{\vartheta})')$. Moreover, since $Z^{n}(\phi) \Rightarrow 0$ in $C_{\infty}(\R)$ for each $\phi \in \Phi$ we have by (II) that  $\tilde{Z}=0$. 
Therefore $Z^{n} \Rightarrow 0$ in $C_{\infty}((\widehat{\Phi}_{\vartheta})')$ and by Lemma \ref{lemmaWeakImpliesUCP} we conclude $Z^{n} \overset{ucp}{\rightarrow} 0$ as $(\widehat{\Phi}_{\vartheta})'$-valued processes. But this last implies that  $Y^{n} \overset{ucp}{\rightarrow} Y$ as $(\widehat{\Phi}_{\vartheta})'$-valued processes. 

From Steps 1-3 we conclude assertions (i) and (ii) in Theorem \ref{theoUCPConver}. The last assertion in Theorem \ref{theoUCPConver} follows since the canonical inclusion from  $(\widehat{\Phi}_{\vartheta})'$ into $\Phi'$ is linear continuous, then as  $\Phi'$-valued processes $Y$  and $Y^{n}$, for $n \in \N$, are continuous processes and $Y^{n} \overset{ucp}{\rightarrow} Y$. 
\end{proof}

\begin{remark}\label{remaUCPCadlagProcesses}
The definition of UCP convergence of a sequence of $\Phi'$-valued processes is equally valid for processes with c\`{a}dl\`{a}g paths. One can check that  Lemma \ref{lemmaWeakImpliesUCP}, Theorem \ref{theoUCPConver} and Theorem \ref{theoWeakConvergCylinProcesInCTInfy} remain valid if we replace the word ``continuous'' by ``c\`{a}dl\`{a}g'',  and if we replace $C_{\infty}(\Phi')$ with the space $D_{\infty}(\Phi')$ of $\Phi'$-valued  c\`{a}dl\`{a}g mappings on $[0,\infty)$ (when equipped with the topology of uniform convergence on compact intervals of time). 
\end{remark}

\subsection{The ultrabornological space setting}\label{subSecUltrobornoUCP}

In this section we show that in Theorem \ref{theoUCPConver} if we
further assume that $\Phi$ is an ultrabornological nuclear, we can show that weak (in the duality sense) UCP convergence implies UCP convergence in $\Phi'$.  

The class of ultrabornological nuclear spaces includes many spaces of functions widely used in analysis. Indeed, it is known (see e.g. \cite{Pietsch, Schaefer, Treves}) that the spaces of test functions $\mathscr{E}_{K} \defeq \mathcal{C}^{\infty}(K)$ ($K$: compact subset of $\R^{d}$), $\mathscr{E}\defeq \mathcal{C}^{\infty}(\R^{d})$, the rapidly decreasing functions $\mathscr{S}(\R^{d})$, and the space of harmonic functions $\mathcal{H}(U)$ ($U$: open subset of $\R^{d}$),  are all  examples of Fr\'{e}chet nuclear spaces. Their (strong) dual spaces $\mathscr{E}'_{K}$, $\mathscr{E}'$, $\mathscr{S}'(\R^{d})$, $\mathcal{H}'(U)$, are also nuclear spaces.
On the other hand, the space of test functions $\mathscr{D}(U) \defeq \mathcal{C}_{c}^{\infty}(U)$ ($U$: open subset of $\R^{d}$), the space of polynomials $\mathcal{P}_{n}$ in $n$-variables, the space of real-valued sequences $\R^{\N}$ (with direct sum topology) are strict inductive limits of Fr\'{e}chet nuclear spaces (hence they are also nuclear). The space of distributions  $\mathscr{D}'(U)$  ($U$: open subset of $\R^{d}$) is also nuclear.   
All the above are examples of (complete) ultrabornological nuclear spaces.


\begin{theorem}\label{theoUCPUltrabornologicalProcesses}
Let $\Phi$ be an ultrabornological nuclear space and let $(X^{n}: n \in \N)$, with $X^{n} =(X^{n}_{t}: t \geq 0)$, be a sequence  of  $\Phi'$-valued continuous processes with Radon probability distributions.  Assume further that for every $\phi \in \Phi$ the sequence $\inner{X^{n}}{\phi}$ converges uniformly on compacts in probability.

Then there exist a weaker countably Hilbertian topology $\vartheta$ on $\Phi$ and a $(\widehat{\Phi}_{\vartheta})'$-valued continuous process  $Y= (Y_{t}: t \geq 0)$  such that 
\begin{enumerate}[label=(\roman*)]
\item For $n \in \N$, $X^{n}$ has an indistinguishable $(\widehat{\Phi}_{\vartheta})'$-valued continuous version.
\item $X^{n} \overset{ucp}{\rightarrow} Y$  as  $(\widehat{\Phi}_{\vartheta})'$-valued processes. 
\end{enumerate}
Moreover, as $\Phi'$-valued processes we have $X^{n} \overset{ucp}{\rightarrow} Y$  and $Y$ has Radon distributions.  
\end{theorem}
\begin{proof}
We will check that assumptions (1)-(3) in Theorem \ref{theoUCPConver} are satisfied for the sequence of induced cylindrical processes. Assumption (3) is part our assumptions. Likewise assumption (1) is immediate since each $X^{n}$ has continuous paths in $\Phi'$. 

To prove assumption (2) observe that because being $\Phi$ ultrabornological it is therefore barrelled, thus for every $n \in N$ and $t \geq 0$ our assumption that $X^{n}_{t}$ has a Radon probability distribution implies that the mapping $X^{n}_{t}: \Phi \rightarrow L^{0} \ProbSpace$ is continuous (see Theorem 2.10 in \cite{FonsecaMora:Existence}). Then for every $n \in \N$ and $T > 0$, 
by Proposition 3.10 in \cite{FonsecaMora:Existence}  the family $( X^{n}_{t}: t \in [0,T] )$ of linear maps from $\Phi$ into $L^{0} \ProbSpace$ is equicontinuous.

By Theorem \ref{theoUCPConver} there exist a weaker countably Hilbertian topology $\vartheta$ on $\Phi$ and  $(\widehat{\Phi}_{\vartheta})'$-valued continuous processes  $Y= (Y_{t}: t \geq 0)$ and $Y^{n}=(Y^{n}_{t}: t \geq 0)$, $n \in N$, such that 
\begin{enumerate}[label=(\Roman*)]
\item For every $\phi \in \Phi$ and $n \in \N$, the real-valued processes $\inner{Y^{n}}{\phi}$ and $\inner{X^{n}}{\phi}$ are indistinguishable.
\item $Y^{n} \overset{ucp}{\rightarrow} Y$  as  $(\widehat{\Phi}_{\vartheta})'$-valued processes. 
\end{enumerate}
To prove (i) in Theorem \ref{theoUCPUltrabornologicalProcesses}, observe that for each $n \in \N$ since $Y^{n}$ and $X^{n}$ are $\Phi'$-valued continuous processes we have by $(I)$ and Proposition 2.12 in \cite{FonsecaMora:Existence} that $X^{n}$ and $Y^{n}$ are indistinguishable. From the above and (II) we conclude (ii) in Theorem \ref{theoUCPUltrabornologicalProcesses}. 
Finally, since $Y$ is a $(\widehat{\Phi}_{\vartheta})'$-valued  process it is regular as a $\Phi'$-valued process. Hence, because $\Phi$ is barrelled $Y$  has  Radon probability distributions by Theorem 2.10 in \cite{FonsecaMora:Existence}. 
\end{proof}

\begin{example}
For each $n \in \N$, let $z^{n}=(z^{n}_{t}: t \geq 0)$ be a $\R^{d}$-valued continuous process. Assume that the sequence $z^{n}$ converges  uniformly on compacts in probability. We will show that there exist  $\mathscr{D}'(\R^{d})$-valued  continuous processes with Radon probability distributions  $Y$ and $Y^{n}$ for $n \in \N$,  such that for every $\phi \in \mathscr{D}(\R^{d})$ and $n \in \N$  the processes $\phi(z^{n})$ and $\inner{Y^{n}}{\phi}$ are indistinguishable, and such that $Y^{n} \overset{ucp}{\rightarrow} Y$ in $\mathscr{D}'(\R^{d})$. 

For each $n \in \N$, define the following cylindrical process in $\mathscr{D}'(\R^{d})$: for each $t \geq 0$, define the map $X^{n}_{t}: \mathscr{D}(\R^{d}) \rightarrow L^{0}\ProbSpace$ by 
$$ X^{n}_{t} (\phi)= \delta_{z^{n}_{t}}(\phi)=\phi(z^{n}_{t}), \quad \forall \, \phi \in \mathscr{D}(\R^{d}), $$
where $\delta_{x}$ denotes the Dirac measure at $x \in \R^{d}$. Each $X^{n}_{t}$ is continuous from $\mathscr{D}(\R^{d})$ into $L^{0}\ProbSpace$ for each $n \in \N$ and $t \geq 0$. Moreover, since each $\phi \in \mathscr{D}(\R^{d})$ is uniformly continuous, then for each $n \in \N$ it is clear that $ X^{n}(\phi)$ has continuous paths and by our assumption we have that $X^{n}(\phi)$ converges  uniformly on compacts in probability.

Since $\mathscr{D}(\R^{d})$ is ultrabornological, by the regularizarion theorem for ultrabornological spaces (Corollary 3.11 in \cite{FonsecaMora:Existence}) each $X^{n}$ has a version $Y^{n}= (Y^{n}_{t}: t \geq 0)$ which is  a $\mathscr{D}'(\R^{d})$-valued  continuous processes with Radon probability distributions. Moreover for every $\phi \in \mathscr{D}(\R^{d})$ the sequence $\inner{Y^{n}}{\phi}$ converges  uniformly on compacts in probability.
 Then by Theorem 
\ref{theoUCPUltrabornologicalProcesses} there exists a $\mathscr{D}'(\R^{d})$-valued  continuous process with Radon probability distributions $Y= (Y_{t}: t \geq 0)$ such that $Y^{n} \overset{ucp}{\rightarrow} Y$ in $\mathscr{D}'(\R^{d})$. 
\end{example}

\subsection{Applications to UCP Convergence of Semimartingales}\label{subSectSemimar}

Let $\Phi$ be an ultrabornological nuclear space.
In this section we provide sufficient conditions for the existence of an UCP limit to a sequence of $\Phi'$-valued semimartingales. Our results will be of great importance in Section \ref{subSectGoodInteg}. We will need the following terminology. 

We denote by $S^{0}$  the linear space (of equivalence classes) of real-valued semimartingales. Recall that the Emery topology on $S^{0}$ is the topology defined by the F-seminorm:
$$d_{em}(z) = \sup\{ d_{ucp}( h \cdot z) : h \in \mathcal{E}_{1}  \},$$
where  $\mathcal{E}_{1}$ is the collection of all the  real-valued predictable processes of the form 
$ \displaystyle{h=a_{0} \mathbbm{1}_{\{0\}} + \sum_{i=1}^{n-1} a_{i} \mathbbm{1}_{(t_{i}, t_{i+1}]}}$,  
for $0 < t_{1} < t_{2} < \dots < t_{n} < \infty$, $a_{i}$ is an $\mathcal{F}_{t_{i}}$-measurable random variable, $\abs{a_{i}} \leq 1$, $i=1, \dots, n-1$, and $ (h \cdot z)_{t}= a_{0} z_{0}+\sum_{i=1}^{n-1} a_{i} \left( z_{t_{i+1} \wedge t}-z_{t_{i} \wedge t} \right)$.  

We always consider $S^{0}$ equipped with Emery's topology which makes it a complete, metrizable, topological vector space. For further details on the Emery topology see e.g. Section 4.9 in \cite{KarandikarRao}.  

A $\Phi'$-valued process $X=(X_{t}: t \geq 0)$ is called a \emph{$\Phi'$-valued semimartingale} if for every $\phi \in \Phi$ the real-valued process $\inner{X}{\phi}$ is a semimartingale. If $X$ has Radon probability distributions, we have by Proposition 3.12 in \cite{FonsecaMora:Semimartingales} that $X$ has a version with c\`{a}dl\`{a}g paths.

\begin{proposition}\label{propConvergenceSemimartingales}
Let $(X^{n}: n \in \N)$, with $X^{n} =(X^{n}_{t}: t \geq 0)$, be a sequence  of  $\Phi'$-valued c\`{a}dl\`{a}g semimartingales with Radon probability distributions. Assume that for every $\phi \in \Phi$ the sequence $\inner{X^{n}}{\phi}$ converges in $S^{0}$. Then there exists a $\Phi'$-valued c\`{a}dl\`{a}g semimartingale $Y =(Y_{t}: t \geq 0)$ with Radon probability distributions such that $X^{n} \overset{ucp}{\rightarrow} Y$.  
\end{proposition}
\begin{proof}
First, since convergence in Emery's topology is stronger than UCP convergence, then by our hypothesis and Theorem \ref{theoUCPUltrabornologicalProcesses} there exists a $\Phi'$-valued c\`{a}dl\`{a}g process $Y =(Y_{t}: t \geq 0)$ with Radon probability distributions such that $X^{n} \overset{ucp}{\rightarrow} Y$. 

We must check that $Y$ is a $\Phi'$-valued semimartingale. To do this, let $\phi \in \Phi$. Observe that by Proposition \ref{propUCPFiniteDimensionProjections} we have $\inner{X^{n}}{\phi} \overset{ucp}{\rightarrow} \inner{Y}{\phi}$. On the other hand, by hypothesis there exists $z_{\phi} \in S^{0}$ such that $\inner{X^{n}}{\phi} \rightarrow  z_{\phi}$ in $S^{0}$, hence $\inner{X^{n}}{\phi} \overset{ucp}{\rightarrow} z_{\phi}$. By uniqueness of limits in UCP we conclude  that  $\inner{Y}{\phi}$ and $z_{\phi}$ are indistinguishable, therefore $\inner{Y}{\phi} \in S^{0}$. Then $Y$ is a $\Phi'$-valued semimartingale.
\end{proof}

Denote by $\mathcal{M}_{loc}^{c}$ the space of all the real-valued continuous local martingales equipped with the UCP topology.  It is known that $(\mathcal{M}_{loc}^{c}, d_{ucp})$ is a complete, metrizable, topological vector space. 
A $\Phi'$-valued process $X=(X_{t}: t \geq 0)$ is called a \emph{continuous local martingale} if $\inner{X}{\phi} \in \mathcal{M}_{loc}^{c}$ for every $\phi \in \Phi$. If $X$ has Radon probability distributions, we have by Proposition 3.12 in \cite{FonsecaMora:Semimartingales} that $X$ has a version with continuous paths.

\begin{proposition}\label{propConvergenceLocalmartingales}
Let $(X^{n}: n \in \N)$, with $X^{n} =(X^{n}_{t}: t \geq 0)$, be a sequence  of  $\Phi'$-valued continuous local martingales with continuous paths in $\Phi'$ and with Radon probability distributions. Assume that for every $\phi \in \Phi$ the sequence $\inner{X^{n}}{\phi}$ converges uniformly on compacts in probability. Then there exists a $\Phi'$-valued continuous local martingale $Y =(Y_{t}: t \geq 0)$ with  continuous paths in $\Phi'$ and Radon probability distributions such that $X^{n} \overset{ucp}{\rightarrow} Y$.  
\end{proposition}
\begin{proof}
By Theorem \ref{theoUCPUltrabornologicalProcesses} there exists a $\Phi'$-valued continuous process $Y =(Y_{t}: t \geq 0)$ with Radon probability distributions such that $X^{n} \overset{ucp}{\rightarrow} Y$. 
Now by Proposition \ref{propUCPFiniteDimensionProjections} we have $\inner{X^{n}}{\phi} \overset{ucp}{\rightarrow} \inner{Y}{\phi}$. Since $\mathcal{M}_{loc}^{c}$ is complete, we have $\inner{Y}{\phi} \in \mathcal{M}_{loc}^{c}$. Thus $Y$ is a $\Phi'$-valued continuous local martingale.  
\end{proof}

\begin{example}
Let $m^{n}$, $n = 0,1, \cdots$, be real-valued continuous local martingales. To each $m^{n}$ we can associate a $\mathscr{S}'(\R)$-valued continuous local martingale  by means of stochastic integration as follows. 

Given $n \in  \N$, define $X^{n}: \mathscr{S}(\R) \rightarrow \mathcal{M}^{c}_{loc}$ by 
$$X^{n}_{t}(\phi)=\int_{0}^{t} \, \phi \, dm^{n}, \quad \forall \phi \in \mathscr{S}(\R),  t \geq 0,$$
where $\int_{0}^{\cdot} \, \phi \, dm^{n}$ is the stochastic integral of $\phi$ with respect to $m^{n}$.

Each $X^{n}$ is linear by the linearity of the stochastic integral. Moreover, $X^{n}$ is a  continuous operator. To see this, if we take $\phi_{k} \rightarrow \phi$ in $\mathscr{S}(\R)$ then $\phi_{k} \rightarrow \phi$ uniformly on $[0,\infty)$, hence by the stochastic dominated convergence theorem (see e.g. Theorem 4.50 in \cite{KarandikarRao}) we have $X^{n}(\phi_{k}) \overset{ucp}{\rightarrow}  X^{n}(\phi)$. This shows the continuity of $X^{n}$ from $\mathscr{S}(\R)$ into $\mathcal{M}^{c}_{loc}$. In particular, for each $t \geq 0$ the mapping $X^{n}_{t}: \mathscr{S}(\R) \rightarrow L^{0} \ProbSpace$ is linear continuous. As the space $\mathscr{S}(\R)$ is ultrabornological  by the regularizarion theorem for ultrabornological spaces (Corollary 3.11 in \cite{FonsecaMora:Existence}) each $X^{n}$ has a version $Y^{n}= (Y^{n}_{t}: t \geq 0)$ which is  a $\mathscr{S}'(\R)$-valued   continuous local martingale with continuous paths in $\mathscr{S}'(\R)$ and with Radon distributions.

Assume that $m^{n} \overset{ucp}{\rightarrow}  m^{0}$. Since the UCP topology on $ \mathcal{M}^{c}_{loc}$ coincides with the semimartingale topology (see e.g. Th\'{e}or\`{e}me IV.5 in \cite{Memin:1980}), for every $\phi \in \Phi$ we have by the continuity of the stochastic integral on the integrators  
(see e.g. Theorem 4.109 in \cite{KarandikarRao}) that $X^{n}(\phi) \overset{ucp}{\rightarrow}  X^{0}(\phi)$. Therefore $\inner{Y^{n}}{\phi} \overset{ucp}{\rightarrow}  \inner{Y^{0}}{\phi}$.  
Then by Proposition  \ref{propConvergenceLocalmartingales} there exists a $\mathscr{S}'(\R)$-valued  continuous local martingale with continuous paths in $\mathscr{S}'(\R)$ and with Radon probability distributions $Y= (Y_{t}: t \geq 0)$, such that $Y^{n} \overset{ucp}{\rightarrow} Y$ in $\mathscr{S}'(\R)$.
\end{example}

\section{Topology on the space of good integrators and continuity of the stochastic integral}\label{subSectGoodInteg}

In Sections 4 and 5 in \cite{FonsecaMora:StochInteg} it is developed a theory of stochastic integration with respect to  semimartingales taking values in the dual of a nuclear space $\Phi$. For a given semimartingale $X$, the stochastic integral mapping $H \mapsto \int \, H \, dX$ is not always continuous into the space $S^{0}$. Whenever the stochastic integral mapping is continuous, $X$ is called a good integrator (a precise definition will be given below). The main objective of this section is to introduce a new topology for the space of all good integrators under the assumption that the space $\Phi$ is either a nuclear Fr\'{e}chet space or the strict inductive limit of nuclear Fr\'{e}chet spaces. We will show that the constructed topology is complete and that the corresponding stochastic integral mapping is continuous on the integrators under this topology; our main tool will be the sufficient conditions for UCP convergence given in Section \ref{sectUCP}.  

\subsection{Real-valued stochastic integration}\label{subSecRealStochInteg}

In this section we review the construction and properties of the stochastic integral in \cite{FonsecaMora:StochInteg}. 
We denote by $(S^{0})_{lcx}$  the convexification of $S^{0}$, i.e. the linear space $S^{0}$ equipped with the strongest locally convex topology on $S^{0}$ that is weaker than the Emery topology. Since the Emery topology is not locally convex, the convexified topology on $S^{0}$ is strictly weaker than the Emery topology. 

Denote by $b\mathcal{P}$ the Banach space of all the bounded predictable processes $h : \R_{+} \times \Omega \rightarrow \R$   equipped with the uniform norm $\norm{h}_{u}=\sup_{(r,\omega)} \abs{h(r,\omega)}$. If $h \in b\mathcal{P}$ and $z \in S^{0}$, then $h$ is stochastically integrable with respect to $z$, and its stochastic integral, that we denote by $h \cdot z=(( h \cdot z)_{t}: t \geq 0)$, is an element of $S^{0}$ (see \cite{Protter}, Theorem IV.15). The mapping $(z,h) \mapsto h \cdot z$ from $S^{0} \times b\mathcal{P}$ into $S^{0}$ is bilinear (see \cite{Protter}, Theorem IV.16-7) and separately continuous (see Theorems 12.4.10-13 in \cite{CohenElliott}).  

Let  $\Phi$ be a complete barrelled nuclear space. 
We denote by $b\mathcal{P}(\Phi)$ the space of all $\Phi$-valued processes $H=(H_{t}: t \geq 0)$ with the property that $\inner{f}{H}\defeq \{ \inner{f}{H_{t}(\omega)}: t \geq 0, \omega \in \Omega   \} \in b\mathcal{P}$ for every $f \in \Phi'$. The space $b\mathcal{P}(\Phi)$ is complete   when equipped with the topology generated by the seminorms $H \mapsto \sup_{(t,\omega)} p(H_{t}(\omega))$ where $p$ ranges over a generating family of seminorms for the topology on $\Phi$ (see Section 4.2 in \cite{FonsecaMora:StochInteg}). Recall that a $\Phi$-valued process is called \emph{elementary} if it takes the form  
\begin{equation}\label{eqElemenProcess}
H_{t}(\omega)=\sum_{k=1}^{m} h_{k}(t,\omega) \phi_{k},
\end{equation}
where for $k=1, \cdots, m$ we have $h_{k} \in b\mathcal{P}$ and $\phi_{k} \in \Phi$. By Corollary 4.9 in \cite{FonsecaMora:StochInteg} the collection of all the $\Phi$-valued elementary process is dense in $b\mathcal{P}(\Phi)$.  

We denote by $S^{0}(\Phi')$ the collection of all the $\Phi'$-valued adapted, regular, c\`{a}dl\`{a}g semimartingales.
Let $X=(X_{t}: t \geq 0)$ be a $\Phi'$-valued adapted semimartingale for which the mapping $\phi \mapsto X(\phi)$ is continuous from $\Phi$ into $S^{0}$ (if $\Phi$ is ultrabornological, one can equivalently ask for the probability distribution of each $X_{t}$ to be Radon; see Proposition 3.15 in \cite{FonsecaMora:Semimartingales} and Theorem 2.10 in \cite{FonsecaMora:Existence}). By Theorem 3.7 and Proposition 3.14  in \cite{FonsecaMora:Semimartingales} $X$ has a regular c\`{a}dl\`{a}g version. Hence $X \in S^{0}(\Phi')$. 

Now the stochastic integral with respect to $X \in S^{0}(\Phi')$ is defined as follows: by Theorem 4.10 in \cite{FonsecaMora:StochInteg} for each $H \in b\mathcal{P}(\Phi)$ there exists a real-valued c\`{a}dl\`{a}g adapted semimartingale $\int \, H \, dX$, called the \emph{stochastic integral} of $H$ with respect to $X$, such that:

\begin{enumerate}
\item For every $\Phi$-valued elementary process of the form \eqref{eqElemenProcess} we have 
\begin{equation}\label{eqActionWeakIntegSimpleIntegNuclear}
  \int \, H \, dX =  \sum_{k=1}^{n} \, h_{k} \cdot \inner{X}{\phi_{k}}. 
\end{equation}
\item $\displaystyle{\left(\int H \, dX \right)^{\tau}=\int H \mathbbm{1}_{[0,\tau]} \, dX= \int H \, dX^{\tau}}$, for every stopping time $\tau$.
\item \label{properBilinearity} The mapping $(H,X) \mapsto \int \, H \, dX$ is bilinear. 
\item The mapping $H \mapsto \int \, H \, dX$ is continuous from  $b\mathcal{P}(\Phi)$ into $(S^{0})_{lcx}$. 
\end{enumerate}

The real-valued process $\int \, H \, dX$ is called the \emph{stochastic integral} of $H$ with respect to $X$. Further properties of the stochastic integral can be found in \cite{FonsecaMora:StochInteg}.

\begin{definition}
A $\Phi'$-valued adapted semimartingale $X=(X_{t}: t \geq 0)$ is a \emph{good integrator} if the mapping $\phi \mapsto X(\phi)$ is continuous from $\Phi$ into $S^{0}$ and if the stochastic integral mapping $H \mapsto \int \, H \, dX$  defines a continuous linear mapping from $b\mathcal{P}(\Phi)$ into $S^{0}$. 
We denote by $\mathbbm{S}^{0}(\Phi')$ the collection of all the $\Phi'$-valued semimartingales which are good integrators. 
\end{definition}

If $\Phi=\R$ then $b \mathcal{P}(\R)=b\mathcal{P}$ and by the Bichteler-Dellacherie theorem we have $\mathbbm{S}^{0}(\R)=S^{0}$. However, in the general case for $\Phi$ it is not clear if every $\Phi'$-valued semimartingale is a good integrator in $\Phi'$. There are however many examples of $\Phi'$-valued semimartingales which are good integrators, see for example Proposition 4.12, Corollary 4.13 and Proposition 7.3 in \cite{FonsecaMora:StochInteg}. 

\begin{proposition}
The space $\mathbbm{S}^{0}(\Phi')$ is a linear subspace of  $S^{0}(\Phi')$. 
\end{proposition}
\begin{proof}
Let $X, Y \in \mathbbm{S}^{0}(\Phi')$ and $c \in \R$. Then $cX+Y$ is a $\Phi'$-valued semimartingale and the mapping $\phi \mapsto (cX+Y)(\phi)$ is continuous from $\Phi$ into $S^{0}$

Now by our assumptions the mapping $H \mapsto \left( c \int \, H \, dX + \int \, H \, dY \right)$ is continuous. Moreover for any given $H \in b\mathcal{P}(\Phi)$ we have by the linearity of the stochastic integral on the integrators (property \ref{properBilinearity} above) that $\int \, H \, d(cX+Y)=c \int \, H \, dX +\int \, H \, dY$. Therefore the mapping $H \mapsto \int \, H \, d(cX+Y)$ is continuous and hence $cX+Y \in \mathbbm{S}^{0}(\Phi')$.
\end{proof}

\subsection{Good integrators when $\Phi$ is nuclear Fr\'{e}chet space}\label{subSectGoodIntegFrechet}

Throughout this section we assume that  $\Phi$ is a nuclear Fr\'{e}chet  space and  $q$ is an $F$-seminorm that generates its topology (recall the examples listed in Section \ref{subSecUltrobornoUCP}). In such a case the space $b\mathcal{P}(\Phi)$ is a Fr\'{e}chet space and its topology can be equivalently defined by the $F$-seminorm $H \mapsto \sup_{(t,\omega)} q(H_{t}(\omega))$.


\begin{definition}\label{defiTopologyGoodIntegrators}
Let $d_{em}$ denotes the $F$-seminorm that generates the Emery topology on $S^{0}$. We define a $F$-seminorm $d_{\mathbbm{S}^{0}}$ on the space $\mathbbm{S}^{0}(\Phi')$ as follows: for a good integrator $X=(X_{t}: t \geq 0)$ in $\Phi'$ let
\begin{equation}\label{eqFSeminormEmery}
d_{\mathbbm{S}^{0}}(X)=\sup \left\{  d_{em} \left( \int \, H \, dX \right): H \in b\mathcal{P}(\Phi), \, \sup_{(t,\omega)} q(H_{t}(\omega)) \leq 1   \right\}. 
\end{equation}
We equip the space $\mathbbm{S}^{0}(\Phi')$ of good integrators in $\Phi'$ with the linear metrizable topology so defined by $d_{\mathbbm{S}^{0}}$. We write $X^{n}  \rightarrow X$ in $\mathbbm{S}^{0}(\Phi')$ when $d_{\mathbbm{S}^{0}}(X^{n} - X) \rightarrow 0$.
\end{definition}


\begin{theorem}\label{theoCompleteSpaceIntegrators}
The space $(\mathbbm{S}^{0}(\Phi'), d_{\mathbbm{S}^{0}})$ is complete. 
\end{theorem}
\begin{proof}
Let $(X^{n}: n \in \N)$, with $X^{n} =(X^{n}_{t}: t \geq 0)$, be a Cauchy sequence in $\mathbbm{S}^{0}(\Phi')$. Let $\phi \in \Phi$ and let $H^{\phi}$ to be defined as $ H^{\phi}_{t}(\omega)= \phi$ if $q(\phi)\leq 1$ and $ H^{\phi}_{t}(\omega)= \phi/ q(\phi)$ if $q(\phi)>1$. 
Observe $H^{\phi} \in b\mathcal{P}(\Phi)$ and  $\sup_{(t,\omega)} q(H^{\phi}_{t}(\omega)) \leq 1  $. 

Moreover, by \eqref{eqActionWeakIntegSimpleIntegNuclear} and \eqref{eqFSeminormEmery} we have 
\begin{eqnarray*}
d_{em}(\inner{X^{n}}{\phi}-\inner{X^{m}}{\phi}) 
& = & C(\phi) \, d_{em} \left( \int \, H^{\phi} dX^{n} - \int \, H^{\phi} dX^{m} \right)  \\
& \leq & C(\phi) \, d_{\mathbbm{S}^{0}}(X^{n}-X^{m}),  
\end{eqnarray*}
where $C(\phi)=1$ if  $q(\phi) \leq 1$ and $C(\phi)=q(\phi)$ if $q(\phi) > 1$. 
Therefore, for every $\phi \in \Phi$ we have $(\inner{X^{n}}{\phi}: n \in \N)$ is a Cauchy sequence in $S^{0}$ and hence converges in $S^{0}$. By Proposition \ref{propConvergenceSemimartingales} there exists a $\Phi'$-valued c\`{a}dl\`{a}g semimartingale $Y =(Y_{t}: t \geq 0)$ with Radon probability distributions such that $X^{n} \overset{ucp}{\rightarrow} Y$. 
We must  show that $Y$ is a good integrator. 

Let $J:b\mathcal{P}(\Phi) \rightarrow S^{0}$ be the stochastic integral mapping $J(H)=\int \, H \, dY$ corresponding to $Y$. To show that $Y$ is a good integrator it is enough to show the existence of a continuous linear operator $I: b\mathcal{P}(\Phi) \rightarrow S^{0}$ that coincide with $J$ on a dense subset of $b\mathcal{P}(\Phi)$.

For every $n \in \N$ let $I^{n}: b\mathcal{P}(\Phi) \rightarrow S^{0}$ be the stochastic integral mapping $I^{n}(H)=\int \, H \, dX_{n}$ corresponding to $X^{n}$. Since each $X^{n}$ is a good integrator we have each $I^{n}$ is linear continuous. To prove the existence of the operator $I$ as described in the above paragraph, we will need to show that that $(I^{n}:n \in \N)$ is an equicontinuous subset of $\mathcal{L}(b\mathcal{P}(\Phi), S^{0})$ and that $I^{n}(H) \rightarrow J(H) $ in $S^{0}$ for every $H$ in a dense subset of  $b\mathcal{P}(\Phi)$.

Now to prove that $(I^{n}:n \in \N)$ is equicontinuous notice that by the Banach-Steinhaus theorem (see Theorem 11.9.5 in \cite{NariciBeckenstein}) it is enough to show that $(I^{n}:n \in \N)$ is pointwise bounded on $b\mathcal{P}(\Phi)$. 
Let $H \in b\mathcal{P}(\Phi)$. There exists $M>0$ such that $\sup_{(t,\omega)} q(H_{t}(\omega)) \leq M$. By \eqref{eqFSeminormEmery}  one has 
\begin{equation}\label{eqCauchyGoodIntegra}
 d_{em}(I^{n}(H)-I^{m}(H)) \leq M d_{\mathbbm{S}^{0}}(X^{n}-X^{m}), \quad \forall \, n,m \in \N. 
\end{equation}
Since $(X^{n}: n \in \N)$ is Cauchy in $\mathbbm{S}^{0}(\Phi')$ we immediately conclude from \eqref{eqCauchyGoodIntegra} that $(I^{n}(H): n \in \N)$ is Cauchy in $S^{0}$, thus $(I^{n}(H): n \in \N)$ is bounded in $S^{0}$. Hence, $(I^{n}:n \in \N)$ is an equicontinuous subset of $\mathcal{L}(b\mathcal{P}(\Phi), S^{0})$.

Our next task will be to prove that $I^{n}(H) \rightarrow J(H) $ in $S^{0}$ for each $\Phi$-valued elementary process. 
Let $\phi \in \Phi$ and $h \in b \mathcal{P}$. Because $X^{n} \overset{ucp}{\rightarrow} Y$ then we have by Lemma \ref{lemmaWeakImpliesUCP} that $\inner{X^{n}}{\phi} \overset{ucp}{\rightarrow} \inner{Y}{\phi}$. 
But since $(\inner{X^{n}}{\phi}: n \in \N)$ is a Cauchy sequence in $S^{0}$ we must have $\inner{X^{n}}{\phi} \rightarrow \inner{Y}{\phi}$ in $S^{0}$ by uniqueness of limits in UCP. Then for any given $h \in b\mathcal{P}$ we have by \eqref{eqActionWeakIntegSimpleIntegNuclear} and the continity of the stochastic integral mapping for real-valued semimartingales (see Theorem 12.4.13 in \cite{CohenElliott}) that 
$$ I^{n}(h\phi)=(h \cdot \inner{X^{n}}{\phi}) \rightarrow (h \cdot \inner{Y}{\phi})=J(h\phi) $$
with convergence in $S^{0}$. Then for every $H$ of the simple form \eqref{eqElemenProcess} we have by linearity that $I^{n}(H) \rightarrow J(H) $ in $S^{0}$. 

From the conclusions of the above paragraphs and because the collection of all  $\Phi$-valued elementary process is dense in $b \mathcal{P}(\Phi)$  we have by Corollary 1 in Section 34.3 in (\cite{Treves}, p.356) that there exists a continuous linear operator $I: b\mathcal{P}(\Phi) \rightarrow S^{0}$ that coincide with $J$ for every $\Phi$-valued elementary process and moreover $I^{n}$ converges to $I$ for the topology of compact convergence in $\mathcal{L}(b\mathcal{P}(\Phi), S^{0})$. Therefore $Y$ is a good integrator and $\int \, H \, dX^{n}$ converges to $\int \, H \, dY$ in $S^{0}$ for every $H \in b\mathcal{P}(\Phi)$. 

Our final task is to show $X^{n} \rightarrow Y$ in $\mathbbm{S}^{0}(\Phi')$. Let $a_{n}=\sup_{k \geq 1} d_{\mathbbm{S}^{0}}(X^{n}-X^{n+k})$.  Then $a_{n} \rightarrow 0$ since $(X^{n}: n \geq 1)$ is Cauchy in $\mathbbm{S}^{0}(\Phi')$. Let $H \in b\mathcal{P}(\Phi)$ for which $\sup_{(t,\omega)} q(H_{t}(\omega)) \leq 1$. By \eqref{eqCauchyGoodIntegra} we have 
$ d_{em}\left( \int \, H \, dX^{n} -\int \, H \, dX^{n+k} \right) \leq  a_{n}$, $\forall \, k \geq 1$. 
Then because $ d_{em}\left( \int \, H \, dX^{m} -\int \, H \, dY \right) \rightarrow 0$ as $m \rightarrow \infty$ we conclude 
$$ d_{em}\left( \int \, H \, dX^{n} -\int \, H \, dY \right) \leq a_{n}, \quad \forall H \in b\mathcal{P}(\Phi), \, \sup_{(t,\omega)} q(H_{t}(\omega)) \leq 1, $$
that is $d_{\mathbbm{S}^{0}}(X^{n}-Y) \leq a_{n}$ with $a_{n} \rightarrow 0$. This finishes the proof of completeness for $(\mathbbm{S}^{0}(\Phi'), d_{\mathbbm{S}^{0}})$.
\end{proof}

As an important consequence of the completeness of the space $\mathbbm{S}^{0}(\Phi')$ we obtain the following useful property of the bilinear mapping defined by the stochastic integration. 

\begin{corollary}\label{coroContiStochIntegralMapping}
The bilinear mapping $(X,H) \mapsto \int \, H \, dX$ defined by the stochastic integral is continuous from $ \mathbbm{S}^{0}(\Phi') \times b\mathcal{P}(\Phi)$ (equipped with the product topology) into $S^{0}$. 
\end{corollary}
\begin{proof}
We first prove that the bilinear mapping is separately continuous. 
If $X \in \mathbbm{S}^{0}(\Phi')$, by the definition of good integrator the mapping $H \mapsto \int \, H \, dX$ is continuous from $b \mathcal{P}(\Phi)$ into $S^{0}$. 

If $H \in b \mathcal{P}(\Phi)$. Then there exists $M>0$ such that $\sup_{(t,\omega)} q(H_{t}(\omega)) \leq M$. By \eqref{eqFSeminormEmery} then one has $ d_{em}(\int \, H \, dX) \leq M d_{\mathbbm{S}^{0}}(X)$ and hence the mapping $X \mapsto \int \, H \, dX$ is continuous from $\mathbbm{S}^{0}(\Phi')$ into $S^{0}$.

Now since $b\mathcal{P}(\Phi)$, $\mathbbm{S}^{0}(\Phi')$ and $S^{0}$ are complete metrizable topological vector spaces and the bilinear mapping defined by the stochastic integral is separately continuous then by Corollary 8 in \cite{Swartz:1984} it is indeed continuous.   
\end{proof}

We can alternatively define on $\mathbbm{S}^{0}(\Phi')$ the (Emery) $F$-seminorm:
\begin{equation}\label{eqAlternativeSeminormGoodIntegra}
\delta_{\mathbbm{S}^{0}}(X)=\sup \left\{  d_{ucp} \left( \int \, H \, dX \right): H \in b\mathcal{P}(\Phi), \, \sup_{(t,\omega)} q(H_{t}(\omega)) \leq 1   \right\}. 
\end{equation}
One can check that $\delta_{\mathbbm{S}^{0}}$ defines a metrizable linear topology on $\mathbbm{S}^{0}(\Phi')$. Moreover, since $d_{ucp}(z) \leq d_{em}(z)$ for all $z \in S^{0}$, then 
$\delta_{\mathbbm{S}^{0}}(X) \leq d_{\mathbbm{S}^{0}}(X)$ 
for all $X \in  \mathbbm{S}^{0}(\Phi')$. Hence the topology defined on $\mathbbm{S}^{0}(\Phi')$ by $\delta_{\mathbbm{S}^{0}}$ is weaker than that defined by $d_{\mathbbm{S}^{0}}$. It is not clear if $(\mathbbm{S}^{0}(\Phi'), \delta_{\mathbbm{S}^{0}})$ is complete. This as for the proof of Theorem \ref{theoCompleteSpaceIntegrators} it was indispensable that $(S^{0},d_{em})$ is complete, however the space $(S^{0},d_{ucp})$ is not.
Nevertheless, as we shall prove below there is a subspace of $\mathbbm{S}^{0}(\Phi')$ that is complete for the topology defined by $\delta_{\mathbbm{S}^{0}}$. 

Let $X=(X_{t}: t \geq 0)$ be a $\Phi'$-valued continuous local martingale with continuous paths in $\Phi'$ and with Radon probability distributions. From Proposition 3.6 and Theorem 4.8 in \cite{FonsecaMora:StochInteg} we have $\int H \ dX \in \mathcal{M}_{loc}^{c}$ for every $H \in b\mathcal{P}(\Phi)$. It is known that $\mathcal{M}_{loc}^{c}$ is a closed subspace of $S^{0}$ and moreover that the induced topology on $\mathcal{M}_{loc}^{c}$ coincides with the UCP topology (see \cite{Memin:1980}, Th\'{e}or\`{e}me IV.5). Therefore, $X$ is a good integrator if and only if the stochastic integral mapping $X \mapsto \int H \ dX $ is continuous from $b\mathcal{P}(\Phi)$ into  $\mathcal{M}_{loc}^{c}$. We denote by $\mathbbm{M}_{loc}^{c}(\Phi')$ the collection of all $\Phi'$-valued continuous local martingales that are good integrators. 

\begin{proposition}\label{propContLocaMartiClosedSubspace}
\begin{enumerate}
\item $(\mathbbm{M}_{loc}^{c}(\Phi'), \delta_{\mathbbm{S}^{0}})$ is complete.  
\item $\mathbbm{M}_{loc}^{c}(\Phi')$ is a closed subspace of $(\mathbbm{S}^{0}(\Phi'), d_{\mathbbm{S}^{0}})$.
\end{enumerate}
\end{proposition}
\begin{proof}
Notice that (2) is a consequence of (1) since the topology defined on $\mathbbm{S}^{0}(\Phi')$ by $\delta_{\mathbbm{S}^{0}}$ is weaker than that defined by $d_{\mathbbm{S}^{0}}$. 

To prove (1) let $(X^{n}: n \in \N)$, with $X^{n}=(X^{n}_{t}: t \geq 0)$, be a Cauchy sequence in $(\mathbbm{M}_{loc}^{c}(\Phi'), \delta_{\mathbbm{S}^{0}})$. Arguing as in the proof of Theorem \ref{theoCompleteSpaceIntegrators} we can show from the definition of $\delta_{\mathbbm{S}^{0}}$ that $(\inner{X^{n}}{\phi}: n \in \N)$ converges in $\mathcal{M}_{loc}^{c}$ for each $\phi \in \Phi$. By Proposition \ref{propConvergenceLocalmartingales} there exists a $\Phi'$-valued continuous local martingale $Y =(Y_{t}: t \geq 0)$ with continuous paths in $\Phi'$ and with Radon probability distributions, such that $X^{n} \overset{ucp}{\rightarrow} Y$. 

As in the proof of Theorem \ref{theoCompleteSpaceIntegrators} to prove that $Y$ is a good integrator we can show that for the stochastic integral mapping $J:b\mathcal{P}(\Phi) \rightarrow \mathcal{M}_{loc}^{c}$, $J(H)=\int \, H \, dY$,  there exists a continuous linear operator $I: b\mathcal{P}(\Phi) \rightarrow \mathcal{M}_{loc}^{c}$ that coincide with $J$ on a dense subset of $b\mathcal{P}(\Phi)$. The proof of this fact can be carried out following similar arguments to those used in the proof of Theorem \ref{theoCompleteSpaceIntegrators}, by using that for each $X^{n}$ the stochastic integral mapping $I^{n}: b\mathcal{P}(\Phi) \rightarrow \mathcal{M}_{loc}^{c}$ is continuous, and replacing when required $d_{\mathbbm{S}^{0}}$, $S^{0}$ and $d_{em}$ with $\delta_{\mathbbm{S}^{0}}$, $\mathcal{M}_{loc}^{c}$ and $d_{ucp}$ respectively. 
We leave the details to the reader. 
\end{proof}

\subsection{Good integrators when $\Phi$ is a strict inductive limit of nuclear Fr\'{e}chet spaces} \label{subSectGoodIntegInducLimitFrechet}

Throughout  this section we assume that $\Phi$ is the strict inductive limit of nuclear Fr\'{e}chet spaces (recall the examples listed in Section \ref{subSecUltrobornoUCP}). To be more specific, let $(\Phi_{n}:n \in \N)$ be an increasing sequence of vector subspaces of the vector space $\Phi$ such that $\Phi=\bigcup_{n \in \N} \Phi_{n}$. For each $n \in \N$ let $(\Phi_{n}, \tau_{n})$ be a nuclear Fr\'{e}chet space such that the natural embedding $j_{n,n+1}$ of  $\Phi_{n}$ into $\Phi_{n+1}$ is a topological isomorphism, i.e. the topology induced by $\tau_{n+1}$ on $\Phi_{n}$ coincides with $\tau_{n}$. The inductive topology $\tau_{ind}$ on $\Phi$ is the finest locally convex topology for which the all the canonical inclusion mappings $j_{n}: \Phi_{n} \rightarrow \Phi$, $n \in \N$, are continuous. We write $\Phi = \underset{n \in \N}{\mbox{ind}} \, \Phi_{n}$. 
Since each $\Phi_{n}$ is complete, the space $\Phi$ is complete (Theorem 4.6.4 
in \cite{Jarchow}). 


Our first result characterizes the space of integrands as the strict inductive limit of the sequence of Fr\'{e}chet spaces $(b \mathcal{P}(\Phi_{n}):n \in \N)$. 

\begin{proposition}\label{propIntegrandsInductLimit}
$\displaystyle b \mathcal{P}(\Phi) \simeq \underset{n \in \N}{\mbox{ind}} \, b \mathcal{P}(\Phi_{n})$. 
\end{proposition}
\begin{proof}
The proof is mainly an application of properties of tensor products and inductive limits. In effect, by Theorem 4.8 in \cite{FonsecaMora:StochInteg} we have $ b \mathcal{P}(\Phi) \simeq \Phi \widehat{\otimes}_{\pi} b\mathcal{P}$. Since  $\Phi= \underset{n \in \N}{\mbox{ind}} \, \Phi_{n}$, by Corollary 15.5.4 in \cite{Jarchow}, p.334, we have 
$$  \Phi \widehat{\otimes}_{\pi} b\mathcal{P} \simeq \underset{n \in \N}{\mbox{ind}} \,\left( \Phi_{n} \widehat{\otimes}_{\pi} b\mathcal{P} \right).$$ 
Yet another application of Theorem 4.8 in \cite{FonsecaMora:StochInteg} yields $\displaystyle b \mathcal{P}(\Phi) \simeq \underset{n \in \N}{\mbox{ind}} \, b \mathcal{P}(\Phi_{n})$.
\end{proof}

One can easily check that the natural embedding of $b \mathcal{P}(\Phi_{n})$ into  $b \mathcal{P}(\Phi_{n+1})$ corresponds to the mapping $H \mapsto j_{n,n+1}H=(j_{n,n+1}H(t,\omega): t \geq 0, \omega \in \Omega)$. For simplicity we denote this mapping by $j_{n,n+1}: b \mathcal{P}(\Phi_{n}) \rightarrow b \mathcal{P}(\Phi_{n+1})$. 

Our next objective is to introduce a topology on the space $\mathbbm{S}^{0}(\Phi')$. To do this, we study the sequence of complete metrizable topological vector spaces  $(\mathbbm{S}^{0}(\Phi_{n}'):n \in \N)$, each equipped with the topology introduced in Definition \ref{defiTopologyGoodIntegrators}. Observe that for each $n \in \N$, the dual operator $j'_{n,n+1}: \Phi'_{n+1} \rightarrow \Phi'_{n}$ is surjective and continuous. 

\begin{lemma}\label{lemmaGoodIntegraDecreasing}
$\forall n \in \N$, $\mathbbm{S}^{0}(\Phi_{n+1}') \subseteq \mathbbm{S}^{0}(\Phi_{n}')$, and the inclusion mapping corresponds to $X \mapsto j'_{n,n+1} X=(j'_{n,n+1} X_{t}: t \geq 0)$. 
\end{lemma}
\begin{proof}
Let $X \in \mathbbm{S}^{0}(\Phi_{n+1}') $, we shall prove that $j'_{n,n+1} X \in \mathbbm{S}^{0}(\Phi_{n}')$. 
First, given $\phi \in \Phi_{n}$, we have 
$$ j'_{n,n+1} X(\phi)\defeq \inner{j'_{n,n+1} X}{\phi}=\inner{X}{j_{n,n+1} \phi},$$
is a real-valued semimartingale. Therefore  $j'_{n,n+1} X$ is a $\Phi'_{n}$-valued semimartingale. Second, the continuity of $\phi \mapsto X(\phi)$ from $\Phi_{n+1}$ into $S^{0}$, and the continuity of  $j_{n,n+1}$ from $\Phi_{n}$ into $\Phi_{n+1}$ shows (see the calculation above) that the mapping  
$\phi \mapsto j'_{n,n+1} X(\phi)$ is continuous from $\Phi_{n}$ into $S^{0}$.  

As for our final step we must show that the stochastic integral mapping $I_{j'_{n,n+1} X}: H \mapsto \int \, H \, d(j'_{n,n+1}X)$ is continuous from $b \mathcal{P}(\Phi_{n}) $ into $S^{0}$. By Corollary 1 in Section 34.3 in (\cite{Treves}, p.356) it is enough to show that there exists a linear continuous operator $J: b \mathcal{P}(\Phi_{n}) \rightarrow S^{0}$ which coincides with $I_{j'_{n,n+1} X}$ on a dense subset of $b \mathcal{P}(\Phi_{n}) $. 

First we define such an operator $J$. First, observe that by our assumption the stochastic integral mapping $I_{X}: b \mathcal{P}(\Phi_{n+1}) \rightarrow S^{0}$, $H \mapsto \int \, H \, d(j'_{n,n+1}X)$ is continuous. Second, by Proposition \ref{propIntegrandsInductLimit} (and the remark preceding its proof) the natural inclusion $j_{n,n+1}: b \mathcal{P}(\Phi_{n}) \rightarrow b \mathcal{P}(\Phi_{n+1})$ is linear continuous. Then the mapping  $J \defeq I_{X} \circ j_{n,n+1}: b \mathcal{P}(\Phi_{n}) \rightarrow S^{0}$ is linear continuous. 

Now we show that  $I_{j'_{n,n+1} X}$ and $J$ coincide on the elementary process. Let $H=h \phi$ where $\phi \in \Phi$ and $h \in b\mathcal{P}$. Then, by  \eqref{eqActionWeakIntegSimpleIntegNuclear} we have
\begin{equation}\label{eqEqualElementIntegrandsComposition}
\int \, H \, d(j'_{n,n+1}X)  =  h \cdot \inner{ j'_{n,n+1}X}{\phi} = h \cdot \inner{X}{j_{n,n+1} \phi} \\
 =  \int \, j_{n,n+1} H \, dX. 
\end{equation}
By linearity of the stochastic integral mapping and \eqref{eqEqualElementIntegrandsComposition} we conclude that 
$I_{j'_{n,n+1} X}(H)=J(H)$ for every elementary process $H$. Then, from the arguments given in the above paragraphs $J$ is a continuous extension of  $I_{j'_{n,n+1} X}$ and so the stochastic integral mapping $H \mapsto \int \, H \, d(j'_{n,n+1}X)$ is continuous of $b \mathcal{P}(\Phi_{n}) $ into $S^{0}$. We have shown $j'_{n,n+1} X \in \mathbbm{S}^{0}(\Phi_{n}')$.
\end{proof}

\begin{remark}\label{remarkLimitInductiveTVS}
Let $(\Psi_{n}, k_{n})_{n \in \N}$ be a countable inductive system of locally convex spaces. The topology of inductive limit in  $\Psi=\bigcup_{n \in \N} \Psi_{n}$ can be taken either in the sense of topological vector spaces or of locally convex spaces (see Proposition 6.6.9 in \cite{Jarchow}, p.112). This way, if $\Psi = \underset{n \in \N}{\mbox{ind}} \, \Psi_{n}$ and $\Lambda$ is a topological vector space, a linear mapping $T: \Psi \rightarrow \Lambda$ is continuous if and only if $T \circ k_{n}: \Psi_{n} \rightarrow \Lambda$ is continuous for every $n \in \N$.  
\end{remark}

\begin{theorem}\label{theoGoodIntegratorIntersection}
$\displaystyle \mathbbm{S}^{0}(\Phi')=\bigcap_{n \in \N} \mathbbm{S}^{0}(\Phi'_{n})$. 
\end{theorem}
\begin{proof}
Let $X \in  \mathbbm{S}^{0}(\Phi')$, we shall prove that $j'_{n} X \in \mathbbm{S}^{0}(\Phi'_{n})$ $\forall n \in \N$. This way, we must proof that 
$j'_{n} X(\phi)$ is a real-valued semimartingale for each $\phi \in \Phi_{n}$,  that the mapping $\phi \mapsto j'_{n}X(\phi)=\inner{X}{j_{n} \phi}$ is continuous from $\Phi_{n}$ into $S^{0}$, and that the stochastic integral mapping $H \mapsto \int \, H \, d(j'_{n}X)$ is continuous from $b \mathcal{P}(\Phi_{n}) $ into $S^{0}$. The above can be proved by following the same arguments as in the proof of Lemma \ref{lemmaGoodIntegraDecreasing} with $j'_{n}$ instead of $j'_{n,n+1}$. We leave the details to the reader. 

Assume now that $X$ is a $\Phi'$-valued process with the property that $j'_{n} X \in \mathbbm{S}^{0}(\Phi'_{n})$ $\forall n \in \N$. We will show that $X \in \mathbbm{S}^{0}(\Phi')$. Let $\phi \in \Phi$.  Since $\Phi=\bigcup_{n \in \N} \Phi_{n}$, there exists $m \in \N$ such that $\phi \in \Phi_{m}$. As $j'_{m} X \in \mathbbm{S}^{0}(\Phi'_{m})$, we have $\inner{X}{\phi}=\inner{j'_{n}X}{\phi} \in S^{0}$. 

As the next step we must show that the mapping $\phi \mapsto X(\phi)=\inner{X}{\phi}$ is continuous from $\Phi$ into $S^{0}$. Since $\Phi=\underset{n \in \N}{\mbox{ind}} \,\Phi_{n}$ and by Remark \ref{remarkLimitInductiveTVS}, it is equivalent to show that for every $n \in \N$ the mapping  $\phi \mapsto X \circ j_{n} (\phi)$ is continuous from $\Phi_{n}$ into $S^{0}$. But this holds true since $\forall n \in \N$, $X \circ j_{n}(\phi)=j'_{n}X(\phi)$ for every $\phi \in \Phi_{n}$ and  $j'_{n} X \in \mathbbm{S}^{0}(\Phi'_{n})$. 

Finally, we must prove that the stochastic integral mapping $I_{X}:b \mathcal{P}(\Phi) \rightarrow S^{0}$, $H \mapsto \int \, H \, d(X)$, is continuous. As before, by Remark \ref{remarkLimitInductiveTVS} it is equivalent to show that for every $n \in \N$ the mapping $I_{X} \circ j_{n}: 
b \mathcal{P}(\Phi_{n}) \rightarrow S^{0}$ is continuous. Let $n \in \N$, as in the proof of Lemma \ref{lemmaGoodIntegraDecreasing}, one can show that 
\begin{equation}\label{eqCompatibleStochIntegInclusionMapping}
I_{j'_{n}X}(H)=\int \, H \, d(j'_{n}X) = \int \, j_{n}H \, dX = I_{X} \circ j_{n}(H), \quad \forall H \in  b \mathcal{P}(\Phi_{n}). 
\end{equation}
In effect, \eqref{eqCompatibleStochIntegInclusionMapping} can be proved as in \eqref{eqEqualElementIntegrandsComposition} replacing $j_{n,n+1}$ with $j_{n}$. Therefore, $I_{X} \circ j_{n}$ is continuous since by our assumption the stochastic integral mapping $I_{j'_{n}X}$ is continuous. We have shown $X \in \mathbbm{S}^{0}(\Phi')$.
\end{proof}

\textbf{Construction of the topology on $\mathbbm{S}^{0}(\Phi')$.} By Lemma \ref{lemmaGoodIntegraDecreasing} and Theorem \ref{theoGoodIntegratorIntersection}  we can equip the space $\mathbbm{S}^{0}(\Phi')$ with the projective limit topology determined by the projective system $(\mathbbm{S}^{0}(\Phi'_{n}), j'_{n,n+1}, j'_{n})_{n \in \N}$ of complete metrizable topological vector spaces. By Proposition 2.8.3 and Corollary 3.2.7 in \cite{Jarchow} we have $\mathbbm{S}^{0}(\Phi')= \underset{n \in \N}{\mbox{proj}} \, \mathbbm{S}^{0}(\Phi'_{n})$ is a complete metrizable topological vector space and a generating family  of $F$-seminorms for its topology is the collection $(d^{n}_{\mathbbm{S}^{0}}: n \in \N)$, where for $n \in \N$, 
$$ d^{n}_{\mathbbm{S}^{0}}(X)=\sup \left\{  d_{em} \left( \int \, H \, d(j'_{n} X) \right): H \in b\mathcal{P}(\Phi_{n}), \, \sup_{(t,\omega)} q_{n}(H_{t}(\omega)) \leq 1   \right\}. $$
and $q_{n}$ is an $F$-seminorm that generates the topology on $\Phi_{n}$.

\begin{corollary}\label{coroContiBilinearMappingLFNuclearSpace}
The bilinear mapping $(X,H) \mapsto \int \, H \, dX$ defined by the stochastic integral is continuous from $ \mathbbm{S}^{0}(\Phi') \times b\mathcal{P}(\Phi)$ (equipped with the product topology) into $S^{0}$. 
\end{corollary}
\begin{proof} By the arguments used in the proof of Corollary \ref{coroContiStochIntegralMapping}, it is enough to show that for a given $H \in b \mathcal{P}(\Phi)$, the mapping $X \mapsto \int \, H \, dX$ is continuous from $\mathbbm{S}^{0}(\Phi')$ into $S^{0}$. 

Let $H \in b \mathcal{P}(\Phi)$. By Proposition \ref{propIntegrandsInductLimit} there exists $m \in \N$ and $\tilde{H} \in  b \mathcal{P}(\Phi_{m})$ for which 
$H=j_{m} \tilde{H}$. Then there exists $M>0$ such that $\sup_{(t,\omega)} q_{m}(\tilde{H}_{t}(\omega)) \leq M$. Hence we have by \eqref{eqCompatibleStochIntegInclusionMapping} and the definition of $d^{m}_{\mathbbm{S}^{0}}$ that
$$ d_{em}\left( \int \, H \, dX \right) = d_{em}\left( \int \,  \tilde{H} \, d(j'_{m}X) \right) \leq M d^{m}_{\mathbbm{S}^{0}}(X).$$ Therefore, the mapping $X \mapsto \int \, H \, dX$ is continuous.
\end{proof}

\section{UCP convergence of solutions to SPDEs with semimartingale noise}\label{sectUCPConvSPDEs}

Through this section, $\Phi$ denotes a complete bornological (hence ultrabornological) nuclear space whose strong dual $\Phi'$ is complete and nuclear. These properties are satisfied by all the examples listed in Section \ref{subSecUltrobornoUCP}. 

\subsection{SPDEs with semimartingale noise}

We denote by $\mathcal{H}^{1}_{S}$ the Banach space (see Section 16.2 in \cite{CohenElliott}) of all the  real-valued semimartingales $x=(x_{t}:t \geq 0)$ satisfying 
$$\norm{x}_{\mathcal{H}^{1}_{S}} \defeq \inf \left\{ \Exp  [m,m]_{\infty}^{1/2} + \Exp \int_{0}^{\infty} \abs{d a_{s} } : x=m+a \right\}<\infty, $$
where the infimum is taken over all the decompositions $x=m+a$ as a sum of a local martingale $m$ and a process of finite variation $a$, and  $([m,m]_{t}: t \geq 0)$ denotes the quadratic variation process associated to the local martingale $m$. 

Let $Z=(Z_{t}: t \geq 0)$  be a  $\mathcal{H}^{1}_{S}$-valued semimartingale in $\Phi'$, i.e.  $\forall \phi \in \Phi$, $Z(\phi) \in \mathcal{H}^{1}_{S}$. We further assume that the mapping $Z: \Phi \rightarrow S^{0}$ is continuous, then by Proposition 4.12 in \cite{FonsecaMora:StochInteg} we have that $Z$ is a good integrator. 

Let $A$ denotes the generator of a strongly continuous $C_{0}$-semigroup $(S(t):t \geq 0) \subseteq \mathcal{L}(\Phi,\Phi)$ (see \cite{Komura:1968} for information on $C_{0}$-semigroups in locally convex spaces). Consider the following class of stochastic evolution equations
$$ dY_{t}=A'Y_{t} dt+dZ_{t}, \quad t \geq 0, $$
with initial condition $Y_{0}=\eta$ $\Prob$-a.e. where $\eta$  is a $\mathcal{F}_{0}$-measurable regular random variable. A \emph{weak solution} to the above equation is a $\Phi'$-valued regular adapted process $Y=(Y_{t}: t \geq 0)$ satisfying that for any given $t > 0$ and  $\phi \in \mbox{Dom}(A)$ we have $\int_{0}^{t} \abs{\inner{Y_{r}}{A\phi}} dr < \infty$ $\Prob$-a.e. and 
\begin{equation*}\label{eqWeakSoluStochEvolEqua}
\inner{Y_{t}}{\phi}= \inner{\eta}{\phi} + \int_{0}^{t} \inner{Y_{r}}{A\phi} dr + \inner{Z_{t}}{\phi}. 
\end{equation*}
If we further assume that $Z$ has continuous paths and $A \in \mathcal{L}(\Phi,\Phi)$, then by Theorem 6.11 in \cite{FonsecaMora:StochInteg} there exists a unique weak solution $(Y_{t}: t \geq 0)$ which is regular and has continuous paths, and satisfies $\Prob$-a.e.
$$
\inner{Y_{t}}{\phi}= \inner{\eta -Z_{0}}{S(t)\phi}  +\int_{0}^{t} \inner{Z_{s}}{S(t-s)A\phi} ds +\inner{Z_{t}}{\phi}, \quad \forall \, t \geq 0, \, \phi \in \Phi.  
$$

\subsection{Convergence of solutions}

For each $n=0,1,2,\cdots$, let $Z^n=(Z^{n}_{t}: t \geq 0)$  be a  $\mathcal{H}^{1}_{S}$-valued semimartingale with continuous paths for which the mapping $Z^{n}: \Phi \rightarrow S^{0}$ is continuous,  $(S^{n}(t):t \geq 0)$ a strongly continuous $C_{0}$-semigroup with generator $A^{n} \in  \mathcal{L}(\Phi,\Phi)$, $\eta^{n}$ a  $\mathcal{F}_{0}$-measurable regular random variable. 
As mentioned in the previous section, the stochastic evolution equation
$$ dY^{n}_{t}=(A^{n})'Y^{n}_{t} dt +dZ^{n}_{t}, \quad Y^{n}_{0}=\eta^n, $$
has unique weak solution $(Y^{n}_{t}: t \geq 0)$ which is regular and has continuous paths, and satisfies $\Prob$-a.e. $\forall \, t \geq 0$,  $\phi \in \Phi$,
$$
\inner{Y^{n}_{t}}{\phi}= \inner{\eta^{n} -Z^{n}_{0}}{S^{n}(t)\phi}  +\int_{0}^{t} \inner{Z^{n}_{s}}{S^{n}(t-s)A^{n}\phi} ds +\inner{Z^{n}_{t}}{\phi}.  
$$

Sufficient conditions for ucp convergence of $Y^{n}$ to $Y^{0}$ are given below: 

\begin{theorem}\label{theoSuffiUCPCongSoluSPDEs}
Assume the following:
\begin{enumerate}
\item \label{assuIniCondBounded} $\eta^{0}$ and $Z^{0}_{0}$ have a bounded image in $\Phi'$ for  $\Prob$-a.e. $\omega \in \Omega$.  
\item \label{assuConvSequEta} $\eta^{n} \rightarrow \eta^{0}$ in probability in $\Phi'$ as $ n\rightarrow \infty$. 
\item \label{assuConvSemig} $\forall \phi \in \Phi$, $S^{n}(t)\phi \rightarrow S^{0}(t)\phi$ as $n \rightarrow \infty$ uniformly in $t$ on any bounded interval of time. 
\item  \label{assuUCPZ} $\forall \phi \in \Phi$, $\inner{Z^{n}}{\phi} \overset{ucp}{\rightarrow}  \inner{Z^{0}}{\phi}$. 
\end{enumerate}
Then $Y^{n} \overset{ucp}{\rightarrow}  Y^{0}$. 
\end{theorem}
\begin{proof} In view of Theorem \ref{theoUCPUltrabornologicalProcesses}, to show $Y^{n} \overset{ucp}{\rightarrow}  Y^{0}$ it sufficies to show $\inner{Y^{n}}{\phi}  \overset{ucp}{\rightarrow}  \inner{Y^{0}}{\phi}$ for every $\phi \in \Phi$.

Given $\phi \in \Phi$, $t \geq 0$, we have 
\begin{eqnarray*}
\inner{Y^{n}_{t}-Y^{0}_{t}}{\phi} 
& = &  \inner{\eta^{n}-Z^{n}_{0}}{S^{n}(t)\phi}- \inner{\eta^{0}-Z^{0}_{0}}{S^{0}(t)\phi}  \\ 
& {} & + \int_{0}^{t} \inner{Z^{n}_{s}}{S^{n}(t-s)A^{n}\phi} ds - \int_{0}^{t} \inner{Z^{0}_{s}}{S^{0}(t-s)A^{0}\phi} ds \\
& {} & + \inner{Z^{n}_{t}- Z^{0}_{t}}{\phi} .
\end{eqnarray*}
Let $T>0$. We must check that the three terms on the right-hand side of the above equality all converge uniformly on $[0,T]$ in probability to $0$ as $n \rightarrow \infty$.

Let $\epsilon>0$. By \ref{assuUCPZ} we have $ \Prob \left( \sup_{0 \leq t \leq T} \abs{\inner{Z^{n}_{t}-Z^{0}_{t}}{\phi}} > \epsilon \right) \rightarrow 0$ as $n \rightarrow \infty$. 

Now, observe that
\begin{flalign*}
& \Prob \left( \sup_{0 \leq t \leq T} \abs{\inner{\eta^{n}-Z^{n}_{0}}{S^{n}(t)\phi}-\inner{\eta^{0}-Z^{0}_{0}}{S^{0}(t)\phi} } > \epsilon \right) \\
&  \leq \Prob \left( \sup_{0 \leq t \leq T} \abs{\inner{\eta^{n}-\eta^{0}}{S^{n}(t)\phi}} > \epsilon \right)  + \Prob \left( \sup_{0 \leq t \leq T} \abs{\inner{Z^{n}_{0}-Z^{0}_{0}}{S^{n}(t)\phi}} > \epsilon \right) \\
& + \Prob \left( \sup_{0 \leq t \leq T} \abs{\inner{\eta^{0}}{S^{n}(t)\phi-S^{0}(t)\phi}} > \epsilon \right)  + \Prob \left( \sup_{0 \leq t \leq T} \abs{\inner{Z^{0}_{0}}{S^{n}(t)\phi-S^{0}(t)\phi}} > \epsilon \right).
\end{flalign*}
We must check that the four terms on the right-hand side of the above equality all converge uniformly on $[0,T]$ in probability to $0$ as $n \rightarrow \infty$. 

Let $B=\bigcup_{n}\bigcup_{0 \leq t \leq T} S^{n}(t)\phi$, which is bounded in $\Phi$ by \ref{assuConvSemig}. Thus, $q_{B}$ given by $q_{B}(f)=\sup_{\psi \in B} \abs{ \inner{f}{\psi}}$ is a continuous seminorm on $\Phi$. By \ref{assuConvSequEta}, as $n \rightarrow \infty$ we have
$$ \Prob \left( \sup_{0 \leq t \leq T} \abs{\inner{\eta^{n}-\eta^{0}}{S^{n}(t)\phi}} > \epsilon \right)  \leq \Prob \left(  q_{B}(\eta^{n}-\eta^{0}) > \epsilon \right) \rightarrow 0,$$
and by \ref{assuUCPZ}, as $n \rightarrow \infty$ we have
$$ \Prob \left( \sup_{0 \leq t \leq T} \abs{\inner{Z^{n}_{0}-Z^{0}_{0}}{S^{n}(t)\phi}} > \epsilon \right) \leq  \Prob \left(  q_{B}(Z^{n}_{0}-Z^{0}_{0}) > \epsilon \right) \rightarrow 0.$$

Now, by \ref{assuIniCondBounded} there exists a bounded $K$ of $\Phi'$ such that $\eta^{0}(\omega), Z^{0}_{0}(\omega) \in K$ for $ \Prob$-a.e. $\omega \in \Omega$. Being $\Phi$ reflexive, we have that $p_{K}$ defined by $p_{K}(\psi)=\sup_{f \in K} \abs{\inner{f}{\psi}}$ is a continuous seminorm on $\Phi$. Then, by  \ref{assuConvSemig} as $n \rightarrow \infty$ we have
$$ \Prob \left( \sup_{0 \leq t \leq T} \abs{\inner{\eta^{0}}{S^{n}(t)\phi-S^{0}(t)\phi}} > \epsilon \right) \leq  
\Prob \left( \sup_{0 \leq t \leq T} p_{K}\left( S^{n}(t)\phi-S^{0}(t)\phi \right)  > \epsilon \right) \rightarrow 0, $$
and
$$ \Prob \left( \sup_{0 \leq t \leq T} \abs{\inner{Z^{0}_{0}}{S^{n}(t)\phi-S^{0}(t)\phi}} > \epsilon \right) \leq  
\Prob \left( \sup_{0 \leq t \leq T} p_{K}\left( S^{n}(t)\phi-S^{0}(t)\phi \right)  > \epsilon \right) \rightarrow 0.$$

As for our final step, observe that 
\begin{flalign*}
& \Prob \left( \abs{  \int_{0}^{t} \inner{Z^{n}_{s}}{S^{n}(t-s)A^{n}\phi} ds - \int_{0}^{t} \inner{Z^{0}_{s}}{S^{0}(t-s)A^{0}\phi} ds } > \epsilon \right)  \\
& \leq  \Prob \left(   \int_{0}^{t} \abs{ \inner{Z^{n}_{s}-Z^{0}_{s}}{S^{n}(t-s)A^{n}\phi} } ds   > \epsilon \right) \\
& +\Prob \left(   \int_{0}^{t} \abs{ \inner{Z^{0}_{s}}{S^{n}(t-s)A^{n}\phi- S^{0}(t-s)A^{0}\phi} } ds   > \epsilon \right). 
\end{flalign*} 

Let $D=\bigcup_{n}\bigcup_{0 \leq s \leq  t \leq T} S^{n}(t-s)A^{n} \phi$, which is bounded in $\Phi$ by \ref{assuConvSemig}. 
We therefore have $q_{D}$ given by $q_{D}(f)=\sup_{\psi \in D} \abs{ \inner{f}{\psi}}$ is a continuous seminorm on $\Phi$, and by \ref{assuUCPZ}, as $n \rightarrow \infty$ we have
$$ \Prob \left(   \int_{0}^{t} \abs{ \inner{Z^{n}_{s}-Z^{0}_{s}}{S^{n}(t-s)A^{n}\phi} } ds   > \epsilon \right) \leq  \Prob \left(  \sup_{0 \leq t \leq T}  q_{D}(Z^{n}_{t}-Z^{0}_{t}) > \epsilon/T \right) \rightarrow 0.$$
Now, since $Z^{0}$ is a  a  $\mathcal{H}^{1}_{S}$-valued semimartingale with continuous paths for which the mapping $Z^{0}: \Phi \rightarrow S^{0}$ is continuous, by Theorem 3.22 in \cite{FonsecaMora:Semimartingales} there exists a continuous Hilbertian seminorm $\rho$ on $\Phi$ such that $Z^{0}$ has a $\Phi'_{\rho}$-valued continuous version. Therefore, there exists $C=C(\rho,T)>0$ such that $\sup_{0 \leq t \leq T} \rho'(Z_{t}^{0}) \leq C$ $\Prob$-a.e.   
Then, by  \ref{assuConvSemig} as $n \rightarrow \infty$ we have
\begin{multline*}
\Prob \left(   \int_{0}^{t} \abs{ \inner{Z^{0}_{s}}{S^{n}(t-s)A^{n}\phi- S^{0}(t-s)A^{0}\phi} } ds   > \epsilon \right) \\
 \leq \Prob \left( \sup_{0 \leq s\leq t \leq T} \rho \left( S^{n}(t-s)A^{n}\phi- S^{0}(t-s)A^{0}\phi \right) > \epsilon /TC \right)  \rightarrow 0.
\end{multline*}
\end{proof}

\begin{remark}
In \cite{PerezAbreuTudor:1992} the authors study UCP convergence of solutions to linear stochastic evolution equations driven by square integrable martingale noise and under the assumption that $\Phi$ is Fr\'echet nuclear. Theorem \ref{theoSuffiUCPCongSoluSPDEs} constitutes an extension of the result in \cite{PerezAbreuTudor:1992} when $\Phi$ is  ultrabornological and for the semimartingale noise setting. 
\end{remark}



\smallskip

\noindent \textbf{Acknowledgements}  This work was supported by The University of Costa Rica through the grant ``821-C2-132- Procesos cil\'{i}ndricos y ecuaciones diferenciales estoc\'{a}sticas''.







\end{document}